\documentclass[a4paper,12pt,oneside,reqno]{amsart}
\usepackage{amssymb,amsfonts,amsmath,amsthm,amscd, bbm}
\usepackage{graphicx, color}
\numberwithin{equation}{section}  

\newcommand{\beq}{\begin{equation}} 
\newcommand{\eeq}{\end{equation}} 
\newcommand{\bea}{\begin{aligned}}
\newcommand{\eea}{\end{aligned}}
\newcommand{\bdm}{\begin{displaymath}}
\newcommand{\edm}{\end{displaymath}}
\newcommand{\barr}{\begin{array}}
\newcommand{\earr}{\end{array}}
\newcommand{\ben}{\begin{enumerate}}
\newcommand{\een}{\end{enumerate}}
\newcommand{\bde}{\begin{description}}
\newcommand{\ede}{\end{description}}

\newtheorem{teor}{Theorem}
\newtheorem{prop}[teor]{Proposition}  
\newtheorem{lem}[teor]{Lemma}  
  
\newtheorem{fact}{Fact}

\newtheorem{Def}[teor]{Definition}  
  
\newtheorem{conj}[teor]{Conjecture}

\newcommand{\R}{\mathbb{R}}

\newcommand{\N}{\mathbb{N}}

\newcommand{\PP}{\mathbb{P}}
\newcommand{\E}{{\mathbb{E}}}

\newcommand{\defi}{\equiv} 

\newcommand{\be}{\beta}

\newcommand{\de}{\delta}

\newcommand{\s}{\sigma}
\newcommand{\vare}{\varepsilon}

\newcommand{\RM}[1]{\MakeUppercase{\romannumeral #1{.}}}
\DeclareMathOperator{\Av}{Av}

\begin{document}
\title[Replica symmetry of ISP]
{On the replica symmetry phase of \\ the independent set problem}

\author[N. Kistler]{Nicola Kistler}
\address{J.W. Goethe-Universit\"at Frankfurt, Germany.}
\email{kistler@math.uni-frankfurt.de, mschmidt@math.uni-frankfurt.de}

\author[M. Schmidt]{Marius A. Schmidt}

\subjclass[2000]{05C80, 82B44, 60K35} \keywords{random graphs, independent set problem, disordered systems,statistical
mechanics}

\thanks{ }

 \date{\today}

\begin{abstract} 
The independent set problem, ISP for short, asks for the maximal number of vertices in a (large) graph which can be occupied such that none of them are neighbors. We address the question from a statistical mechanics perspective, in the case of Erd\H{o}s-R\'{e}nyi random graphs. We thereby introduce a Hamiltonian penalizing configurations which do not satisfy the non-neighboring constraint:  the ground state of the ensuing disordered system corresponds to the solution of the ISP. Identifying the ground state amounts, in turns, to control the phase where replica symmetry is broken, which is way beyond our current understanding. By means of Talagrand's cavity method, we rigorously establish the existence of a replica symmetry phase, computing, in particular, the free energy in the limit of large graphs. A conjectural  formula for the ground state, hence for the solution of the ISP, is also derived. Being based on the Parisi theory, the emerging picture is that of a staggering complexity. 
\end{abstract}

\maketitle

\tableofcontents

\section{Introduction}

The ISP is a fundamental question in computer science, see e.g. \cite{aco_eft, bandy_bamarnik,  dani_moore,     frieze, frieze_mcdiarmid, gamarnik_et_al} and references therein. {\it Given a graph, what is the largest fraction of vertices which can be occupied such that none of them are neighbors}? With applications in mind where the graph is large, we address here the question in the case of the paradigmatical Erd\H{o}s-R\'{e}nyi random graph $\mathbf G_{N, p}$, i.e. the complete graph on $N$ vertices where each edge is retained with probability $p$ independently of each other. We are interested in the ISP for a given realization of the graph $\mathbf G_{N, p}$ in the large $N$-limit. 

To formalize, we consider a configuration space $\Sigma_N \defi \{0, 1\}^N$. Given a configuration $\s = (\s_1, \dots, \s_N) \in \Sigma_N$ we refer to $\s_i$ as the spin at site $i$. We say that site $i$ is occupied if $\s_i = 1$, and unoccupied otherwise. Consider then random variables $\{g_{ij}, 1 \leq i < j \leq N\}$ on some probability space $(\Omega, \mathcal F, \PP)$; these are assumed to be independent, Bernoulli-distributed with success probability $\gamma/N$. (Expectation w.r.t. such random variables will be denoted by $\E$). Site $i$ and $j$ are neighbors if $g_{ij} = 1$. This construction thus corresponds to the ISP on $\mathbf G_{N, p}$, where $p \defi \gamma/N$. In other words, the parameter $\gamma$ measures the amount of {\it dilution}: the larger it gets, the more connected the underlying random graph. Finally, we consider the random function $H_N: \Sigma_N \to \N$ defined as 
\beq \label{hardcore}
H_{N}(\s) \defi \begin{cases}
\sum_{i=1}^N \s_i  & \text{if} \; \sum_{1 \leq i < j \leq N} g_{ij} \s_i \s_j =  0, \\
- \infty & \text{otherwise}.
\end{cases}
\eeq
Thus, the largest fraction of non-neighboring sites which can be occupied is, on average, 
\beq
\E\left[ \frac{1}{N}  \max_{\s \in  \Sigma_N} H_N(\s) \right] \defi \texttt{ISP}_N(\gamma).
\eeq

In the form given above, the non-neighboring condition is a hard-core constraint which makes the problem 
all the more challenging. In this paper, we adopt a statistical mechanics perspective. We refer the reader to the lecture notes of Montanari \cite{montanari} for an excellent exposition of this point of view, the relation with combinatorial problems, as well as relevant references (see also, e.g., \cite{desanctis_guerra}). Precisely, we introduce the {\it Hamiltonian} 
\beq \label{ham}
H_{N, \be, h, \gamma}(\s) \defi h \sum_{i\leq N} \s_i - \be \sum_{1\leq i<j\leq N} g_{ij} \s_i \s_j, 
\eeq
where $h, \be \geq 0$ are, respectively, the external magnetic field, and the inverse of temperature. The associated {\it Gibbs measure}  is then
\beq
\mathcal G_{N, \be, h, \gamma}(\s) \defi \frac{\exp H_{N, \be, h, \gamma}(\s)}{Z_N(\be, h, \gamma)}, \quad \s \in \Sigma_N, 
\eeq
where 
\beq\label{part_func}
Z_N(\be, h, \gamma) \defi \sum_{\s\in \Sigma_N} \exp H_{N, \be, h, \gamma}(\s)
\eeq 
is the {\it partition function}. Remark that the Gibbs measure is a random ("quenched") probability measure on $\Sigma_N$, the randomness stemming from the $g$-disorder.  The ensuing disordered system may be seen as a soft version of the ISP:  the configurations not satisfying the non-neighboring condition, although not suppressed, are exponentially penalized. Intuitively, the Gibbs measure will thus charge, for large $\be$, only configurations which 'overwhelmingly' satisfy the hard-core constraint.  For {\it finite} $N$, this intuition is indeed correct: the largest fraction of sites which can be occupied on average may be recovered from the mean {\it free energy}
\beq \label{free_energ}
f_N(\be, h, \gamma) \defi \E\left[ \frac{1}{N} \log Z_N(\be, h, \gamma) \right],
\eeq
for then it clearly holds that 
\beq 
 \lim_{h \to \infty}\lim_{\be \to \infty} \frac{f_N(\be, h, \gamma)}{h}   = \texttt{ISP}_N(\gamma). 
\eeq
Therefore, assuming that one can {\it i)} compute the limiting free energy for all $\beta, h$ and $\gamma$, and {\it ii)} justify the interchange of $(\be, h)$- and $N$-limit, the statistical mechanics approach would yield the solution of 
the ISP in the case of {\it infinite} Erd\H{o}s-R\'{e}nyi graphs. The first issue is the crux of the method, while the second may be considered a technical, albeit challenging, difficulty. In fact, computing the low temperature (large $\be$) limit of disordered systems is a notorious problem which leads into the realm of replica symmetry breaking \cite{parisi}, a phenomenon that remains to these days rather perplexing (we will dwell on this in Section \ref{facconj} below). 

As a first, modest step we tackle here the phase of replica symmetry, computing, in particular, the large-$N$ free energy in the high temperature regime  (small $\be$), or low connectivity (small $\gamma$). This is done by an adaptation of Talagrand's  cavity method \cite{talagrand}, to date the most powerful, and flexible tool to address the replica symmetry phase of a (any?) diluted disordered system of mean field type. 

Finally, we also provide an explicit, albeit conjectural formula the low  temperature free energy (for any $\be, h$ and $\gamma$), hence for the solution of the ISP: the method relies on the interpolation akin to the one first introduced by Guerra in \cite{guerra}  and then  Aizenman-Sims-Starr \cite{ass} for mean filed models, and implemented for diluted models by Franz-Leone \cite{franz} and Panchenko-Talagrand \cite{panchy_tala}.

\section{Main results}

\subsection{The phase of replica symmetry}

The key idea is natural, and simple: for small $\be$ (high temperature), or small $\gamma$ (strong dilution) the Gibbs measure $\E \mathcal G_{N, \be, h, \gamma}$  restricted to a finite number of spins should approach a product measure in the large $N$-limit; assuming that the system settles down to a "steady state", the law of the spins must then satisfy a natural self-consistency. The cavity method implements this insight by integrating out one spin at a time (creating {\it cavities}), thereby showing that the procedure is indeed a contraction. 

To see how this precisely goes, we need some notation. For ease of exposition we will henceforth drop the subscripts in the Hamiltonian, i.e. we write $H(\s)$ for the Hamiltonian $H_{N, \be, h, \gamma}(\s)$ on the $N$-system, and denote by $\left < \right>$ expectation w.r.t. the quenched Gibbs measure. By $H_-(\s)$ we understand the Hamiltonian on the $N-1$ system at parameters $\be, h$ but slightly increased dilution $\gamma' \defi \frac{N-1}{N} \gamma$, and $\left< \right>_{-}$ stands for the associated quenched average over $\Sigma_{N-1}$. Finally, for $Y = (y_i)_{i\in \N}$ with $y_i \in [0,1]$, we denote by $\left< \right>_Y$ the {\it product} measure on spins with marginals given by $\left< \s_i \right>_Y = y_i$.

Let us work out some implications of the intuition that Gibbs measure should resemble a product measure. Under this assumption, and since spins take values in $\{0,1\}$ only, the quenched Gibbs measure is specified by the  "magnetization" of the spins: for instance at site $N$, this reads
\beq \bea \label{mean_one}
\left< \s_N \right> & = \frac{\sum_{\s \in \Sigma_N} \s_N \exp H(\s)}{\sum_{\s \in \Sigma_N} \exp H(\s) }
\eea \eeq
We now write $H(\s) = H_-(\s_1, \dots \s_{N-1}) + \s_N \left( h- \be \sum_{i=1}^{N-1} g_{i, N} \s_i \right)$, perform the trace over $\s_N \in \{0,1\}$, and finally divide both numerator and denominator in \eqref{mean_one} by the partition function on $\Sigma_{N-1}$ associated to  the Hamiltonian $H_-$. This leads to
\beq \label{mean_two}
\left< \s_N \right> = \frac{\left< \exp\left( h- \be \sum_{i=1}^{N-1} g_{i, N} \s_i \right)\right>_-}{1+\left< \exp\left( h- \be \sum_{i=1}^{N-1} g_{i, N} \s_i \right)\right>_-} = \left( 1+\left< \exp\left( h- \be \sum_{i=1}^{N-1} g_{i, N} \s_i \right)\right>_{-}^{-1} \right)^{-1}
\eeq
The sum on the r.h.s. of \eqref{mean_two} is over the (random) set $\{i\leq N-1: g_{i,N} = 1\}$; in the large $N$-limit its cardinality weakly approaches a Poisson random variable of mean $\gamma$, which we denote by $r$.  In other words, the distribution of $\left< \s_N\right>$ should be close to the law of 
\beq \label{inter}
\left( 1+\left< \exp\left( h- \be \sum_{i\leq r} \s_i \right)\right>_Y^{-1} \right)^{-1}\,,
\eeq
where $Y \defi \left( \left< \s_i \right>_- \right)_{i \leq r}$ are the magnetizations on the $(N-1)$-system. Exploiting the product measure property, \eqref{inter} may be written as 
\beq \label{will show up again}
\left( 1+ e^{-h} \prod_{i\leq r} \left( 1-  \left( 1- e^{-\be} \right) \left< \s_i \right>_- \right)^{-1} \right)^{-1} \,.
\eeq
To summarize, we should have 
\beq \label{apr}
\mathcal L\left( \left< \s_N \right> \right) \approx \mathcal L\left( \left( 1+ e^{-h} \prod_{i\leq r} \left( 1-  \left( 1- e^{-\be} \right) \left< \s_i \right>_- \right)^{-1} \right)^{-1} \right) \,,
\eeq
where $\mathcal L$ stands for law. Remark that the r.h.s of \eqref{apr} involves the magnetizations on the $(N-1)$-system, whereas the l.h.s refers to the $N$-system: it seems plausible that there shouldn't be any difference in the large $N$-limit, in which case \eqref{apr} would appear as a natural self-consistency property. 

To rigorously formulate the above line of reasoning we introduce an operator $T = T_{\be, h, \gamma}$ acting on $(\mathcal M, d)$, the space of probability measures on $[0,1]$ equipped with the Monge-Kantorovich distance. The latter is defined as $d(\mu_1, \mu_2) \defi \inf \E \left| X-Y\right|$, where the infimum is taken over all couplings $(X,Y)$ such that $\mathcal L(X)= \mu_1$ and $\mathcal L(Y)= \mu_2$. (The associated convergence is equivalent to the usual weak-convergence of probability measures: $d(\mu_n, \mu) \to 0$ if and only if $\int f d\mu_n \to \int f d\mu$, for all $f$ which are continuous and bounded, see e.g. \cite{talagrand} for details.) 

Given $\nu \in \mathcal M$, we consider an infinite sequence $X = (X_i)_{i\in \N}$ of independent, $\nu$-distributed random variables taking values in $[0,1]$. The aforementioned operator 
$T_{\be, h, \gamma}: \mathcal M \to \mathcal M$ is then 
\beq \label{operator}
T_{\be, h, \gamma}(\nu) \defi \mathcal L\left( \left( 1+\left< \exp\left( h- \be \sum_{i\leq r} \s_i \right)\right>_X^{-1} \right)^{-1} \right)\,.
\eeq
The self-consistency standing behind \eqref{inter} or, which is the same, \eqref{apr}, corresponds then to the fixpoints 
\beq
\nu = T_{\be, h, \gamma}(\nu). 
\eeq
Existence, uniqueness, and properties of solutions to this equation will naturally depend on the underlying parameters $\be, h, \gamma$. Let us set 
\beq \label{replicasymmetry}
C(\be, \gamma) \defi 7\left( \gamma+ \gamma^3 \right)\left( e^{2\be}-1\right)\exp \left( \gamma\left( e^{2\be} - 1 \right)\right)\,.
\eeq
Here is a first result. 

\begin{prop} \label{prop_contraction} With the above notation: 
\begin{itemize}
\item[i)] Assume that $\be, \gamma$ are such that $C(\be, \gamma) < 1$. Then the $T$-operator is a contraction on $(\mathcal M, d)$. In particular, there exists a unique solution $\nu_\star= \nu_\star(\be, h, \gamma)$ of the fixpoint-equation $T_{\be, h, \gamma}(\nu) = \nu$.  
\item[ii)] Assume that $\be, \gamma, \gamma'$ are such that $C(\be,\gamma)<1$ and $C(\be,\gamma')<1$.
For $\nu_\star\left(\be,h,\gamma \right)$ and $\nu_\star\left(\be,h,\gamma' \right)$ solutions of the corresponding fixpoints, the continuity estimate holds: 
\beq \label{Tcontinuous}
d(\nu_\star\left(\be,h,\gamma \right),\nu_\star\left(\be,h,\gamma' \right))\leq \mathcal K_{\be, \gamma, \gamma'} \left| \gamma-\gamma' \right|,
\eeq
where $\mathcal K_{\be, \gamma, \gamma'} \defi \min\left\{  \left( 1-C(\beta, \gamma)\right)^{-1},  \left( 1-C(\beta, \gamma')\right)^{-1}\right\}$.
\end{itemize}
\end{prop}

A cautionary note is compulsive. The requirement $C(\be, \gamma) < 1$ identifies a region of parameters we refer here and throughout as \emph{replica symmetry phase}. It should be however stressed right away that our definition presumably covers only a wee-tiny region of the 'true' replica symmetry phase. It is natural to conjecture that the latter coincides with the {\it largest} region in the $(\be, h, \gamma)$-space where the fixpoint equation admits a unique solution. This guess is however based on nothing more than (some) similarities with models which are (only slightly) better understood than the ISP. 

Next is our main result concerning the Gibbs measure in the replica symmetry phase. It puts on rigorous ground the key insight that finitely many spins "decouple" in the large $N$-limit (provided $\be, \gamma$ are small enough). In order to formulate this precisely we need some notation. 

For any function $f: \R \to \R$ we denote by $\| f\|_\infty \defi \sup_x  |f(x)|$ its supremum norm; for $L \in \R$, we say that $f$ is {\it $L$-Lipschitz} if  $|f(x)- f(y)|\leq L |x-y|$ for any $x,y \in \R$ (or any subset on which $f$ is defined).

\begin{teor} \label{full_ind} Let $k,N \in \N$. For any function $f: \{0,1\}^k \rightarrow \R,$  and any function $g$ which is $L_g$-Lipschitz on $\left[\min f, \max f \right]$, there exists a function  $\alpha(\be, \gamma)$ which is increasing in both coordinates (not depending on $k, N, f,g$) and finite for $C(\be, \gamma) < 1$, such that the following is true: 
\[
\left|\E g\left(\left<f\left(\sigma_1,..,\sigma_k\right)\right>\right) - \E g\left(\E\left[f\left(B_1,..,B_k\right) | X\right]\right) \right|\leq \alpha(\beta,\gamma) \| f \|_\infty L_g \frac{k^3}{N} \,,
\]
where $\nu_\star$ is the unique fixpoint of $T_{\be, \gamma, h}$, $X= (X_1,..,X_k)$ is a vector of independent $\nu_\star$-distributed random variables and, given $X$, the $B_i's$ are independent, and Bernoulli($X_i$)-distributed.

\end{teor}

The decoupling of spins in the large $N$-limit plays a fundamental role in the computation of the free energy in the phase of replica symmetry. Here is the upshot.

\begin{teor} \label{free_energy}
Assume that $\be, \gamma$ are such that $C(\be, \gamma) < 1$. Then the limiting free energy 
\[
f(\be, h, \gamma) \defi \lim_{N\to \infty} f_N(\be, h, \gamma)
\]
exists, and is given by 
\[
f(\be, h, \gamma) = \E \log\left( 1+ e^h \prod_{i\leq r} \left(1- X_i \left( 1-e^{-\be}\right) \right) \right)
+ \frac{\gamma}{2} \E \log\left(1-  \left( 1-e^{-\be} \right) X_1 X_2 \right)\,,
\]
where the $X$'s are independent, $\nu_\star$-distributed, and $r$ is Poisson$(\gamma)$-distributed, independent of all $X$.
\end{teor}

The proof of Theorem \ref{free_energy} is given in Section \ref{free_energy_proof}: it relies on the cavity method, i.e. on integrating out one spin at a time, and on Proposition \ref{prop_contraction} and Theorem \ref{full_ind}. The simple proof of  Proposition \ref{prop_contraction}  is given 
in Section \ref{proof_T}. The proof of Theorem \ref{full_ind}, being technically involved, is deferred to Section \ref{proof}. 
Before that, we however briefly discuss what {\it might} happen for large $\be$. Perhaps not surprisingly, the Parisi Theory \cite{giogio} suggests a behavior of stunning intricacy. 

\subsection{The phase of broken replica symmetry: a fact, and a conjecture} \label{facconj}
We unfortunately need an arsenal of notations, concepts and definitions. Let $K \in \N$. Recall that $\mathcal M_1$ stands for the space of probability measures on $[0,1]$. For $i =1 \dots K+1$ we define inductively $\mathcal M_i$ as the space of probability measures on $\mathcal M_{i-1}$. 
\begin{Def}
A measure $\zeta \in \mathcal M_{K+1}$ is a \emph{$K$-level directing measure}. 
\end{Def}
Here is another definition; the reason for the terminology will become clear below. We denote by $\boldsymbol i^{l} \defi (i_1, \dots, i_l) \in \N^{l}$ a multi-index of length $l \in \N$. 
\begin{Def}
The \emph{quenched magnetizations driven by the directing measure $(\zeta, K)$} is a collection of 
random variables $X^{\boldsymbol{i}^{(K)}} \in [0,1]$ which are constructed as follows. Consider first a $\zeta$-distributed random variable, denoted by  $\eta_\emptyset$ and inductively construct the array of random variables $\left( \eta_{\boldsymbol{i}^{(l-1)}, j}^{(l)}, {j \in \N}\right)$ : these are assumed to be independent and $\eta_{{\boldsymbol i}^{l-1}}^{(l-1)}$-distributed. We then set $X^{\boldsymbol{i}^{(K)}} \defi \eta^{{(K)}}_{\boldsymbol{i}^{K}}$.
\end{Def}

We also need to recall the so-called Derrida-Ruelle cascades \cite{ruelle}. These are point process on  
$[0,1]$ with an in-built tree-like (hierarchical) structure. 

\begin{Def} Consider an array ${\boldsymbol m}^K = \left( m_1, m_2, \dots, m_K \right)$ where $0 < m_1 < m_2 <\dots < m_K < 1$. For  any ${\boldsymbol i}^{l-1}$  we denote by $(e^{(l)}_{{\boldsymbol i^{l-1}, j} })_{j \in \N}$ a Poisson point process on $\R_+$ with intensity $t^{-m_l -1} dt$; the point processes $(e^{(l)}_{{\boldsymbol i^{l-1}, j} })_{j \in \N}$ and $(e^{(l)}_{{\boldsymbol k^{l-1}, j} })_{j \in \N}$ are independent as soon as $\boldsymbol i^{l-1}
\neq \boldsymbol k^{l-1}$. The "levels" $e^{(k)}$ and $e^{(l)}$ are also assumed to be independent as soon as $k\neq l$. We define the point process
$E = \left( e_{\boldsymbol i} \right)_{\boldsymbol i \in \N^K}$, where $e_{\boldsymbol i} \defi e^{(1)}_{i_1}\cdot e^{(2)}_{i_1, i_2}\cdots e^{(K)}_{i_1, \dots, i_K}$.
A $K$-levels \emph{Derrida-Ruelle cascade} with parameters $\boldsymbol m$ is the point process 
\[
\mathcal V_{\boldsymbol m} \defi \left( v_{\boldsymbol \delta} \right)_{\boldsymbol \delta \in \N^K}, \quad \text{where}\;\; v_{\boldsymbol \delta} \defi \frac{e_{\boldsymbol \delta}}{\sum_{\boldsymbol \tau \in \N^K} e_{\boldsymbol \tau}}.
\]
\end{Def}
We are not done with definitions. The following will play an absolutely crucial role. 
\begin{Def}
A K-levels \emph{M\'{e}zard-Parisi structure} is a couple $(\zeta, \mathcal V_{\boldsymbol m})$ consisting  of a K-levels directing measure $\zeta$, and a $K$-levels Derrida-Ruelle cascade $\mathcal V_{\boldsymbol m}$ which is independent of $\zeta$.
\end{Def}
\vspace{0.4cm}
Let us assume henceforth to be given an MP-structure $(\zeta, \mathcal V_{\boldsymbol m} = (v_{\boldsymbol \delta})_{\boldsymbol \delta})$. Recalling that $(\be, h, \gamma)$ are the parameters associated to the ISP (inverse of temperature/magnetic field/dilution), and for $t \in [0,1]$, we consider the interpolating Hamiltonian
\[
H^{\boldsymbol \delta}_{N,t}(\sigma) \defi -\beta \sum\limits_{1\leq i<j \leq N} g^*_{i,j} \sigma_i \sigma_j + \sum\limits_{i,j=1}^N \hat{g}_{i,j} \sigma_i \log \left<e^{-\beta\varepsilon}\right>_{X^{\boldsymbol \de}_{i,j}}+ h\sum\limits_{i=1}^N \sigma_i 
\]
where:
\begin{itemize}
\item $\boldsymbol \delta \in \N^K$ is a multi-index.
\item The $g^*$ are Bernoulli$(\gamma t/N)$-distributed, the $\hat{g}$ are Bernoulli$\left(\gamma (1-t)/N\right)$- distributed, all independent. 
\item The $(X^{\boldsymbol \delta}_{i,j})_{\boldsymbol \delta}$ are independent quenched magnetizations driven by the $K$-levels directing measure $\zeta$, all independent, and independent of the $g^*$ as well as the $\hat g$. (The $\vare$ appearing in the logarithm is of course a spin taking values $0$ or $1$).
\end{itemize}
We define the interpolating Gibbs measure on $\Sigma_N \times \N^K$ according to 
\[
 \mathcal{G}_t(\sigma,\boldsymbol \delta) :=  \frac{ v_{\boldsymbol \delta} \exp\left(H^{\boldsymbol \delta}_{N,t}(\sigma)\right)}{\mathcal Z},
\]
where $\mathcal Z$ is the obvious normalization. We write $\left<\right>_{t}$ for expectation w.r.t. $\mathcal{G}_t$, and $\left<\right>_{t}^{\otimes n}$  for expectation w.r.t. $\mathcal{G}_t^{\otimes n}$. We also introduce the "interpolating free energy" 
\[
\varphi(t):= \frac{1}{N}\E \log \sum\limits_{\boldsymbol{\delta}\in \N^K}\sum\limits_{\sigma \in \Sigma_N} v_{\boldsymbol{\delta}} \exp\left(H^{\boldsymbol{\delta}}_{N,t}(\sigma)\right).
\]
and the {\it M\'{e}zard-Parisi functional}
\[
\text{\sf MP}_{\be, h, \gamma}(\zeta, \mathcal V_{\boldsymbol m}) \defi 
\E \log \sum\limits_{\boldsymbol \delta\in \N^K} v_{\boldsymbol \delta} \left(1+ e^h \prod\limits_{j=1}^r\left<e^{-\beta \varepsilon}\right>_{X^{\boldsymbol \delta}_{j}}\right) - \frac{\gamma}{2}\E \log \sum\limits_{\boldsymbol \delta \in \N^K} v_{\boldsymbol \delta }  \left<e^{-\beta\varepsilon_1\varepsilon_2}\right>_{\left(X^{\boldsymbol \delta}_1,X^{\boldsymbol \delta}_2\right)} 
\]
In the above, $r$ is a Poisson$(\gamma)$ random variable which is independent of everything else and the $X$ are the quenched magnetizations driven by the $K$-level directing measure (independent of each other for different subindeces, and "hierarchically dependent" what pertains the superindeces). Remark that the MP-functional does not depend on the size of the system $N$.

With these definitions, by the fundamental theorem of calculus,
\[
\varphi(1) = \varphi(0) + \int_0^1 \varphi'(t) dt\,.
\]
But for $t=1$ the Hamiltonian $H^{\boldsymbol \delta}_{N, t= 1}$ coincides with the original ISP, hence
 \beq \label{int_f}
f_N(\be, h, \gamma) = \varphi(0) + \int_0^1 \varphi'(t) dt\,,
\eeq
since $\sum_{\boldsymbol \de} v_{\delta}  = 1$. Under the light of \eqref{int_f}, it would be useful to get a handle on $\varphi(0)$ and $\varphi'(t)$. The computations behind this step are straightforward, but long: they are postponed to the Appendix. Here we shall simply state the upshot, relating \eqref{int_f} and the PM-functional just introduced. 

\begin{fact} \label{interpolation} For any $N\in \N$, $(\be, h, \gamma)$ and K-levels MP-structure $(\zeta, \mathcal V_{\boldsymbol m})$, it holds: 
\beq \label{free_lowtemp}
f_N(\be, h, \gamma) = \text{\sf{MP}}_{\be, h, \gamma}(\zeta,  \mathcal V_{\boldsymbol m})  + R_{N, \be, h, \gamma}(\zeta, \mathcal V_{\boldsymbol m}), 
\eeq
where the "rest-term" is given by
\beq \label{restterm_hell} \bea 
R_{N, \be, h, \gamma}(\zeta, \mathcal V_{\boldsymbol m}) = -\frac{\gamma}{2}\sum\limits_{n=1}^{\infty} \frac{\left(e^{-\beta}-1\right)^n}{n} \int_0^1 dt \, \E \left<\left(\frac{1}{N}\sum\limits_{i\leq N}\prod\limits_{l=1}^n \sigma^l_i- \E_X\prod\limits_{l\leq n}X^{\delta^l}\right)^2 \right>^{\otimes n}_t.
\eea \eeq
\end{fact}
(The superindeces on the r.h.s. above, i.e. those in $\s^l, l = 1 \dots n$, refer to configurations drawn from $\mathcal G_t^{\otimes n}$.) We can now finally move to 

\begin{conj} \label{conj_hell} To given $(\be, h, \gamma)$ there exists a unique $K^\star$-levels MP-structure $
\mathcal S^\star = \mathcal S^\star(\be, h, \gamma) \defi (\zeta^\star, \mathcal V^\star_{\boldsymbol m})$ such that 
\[
\lim_{N\to \infty} R_{N, \be, h, \gamma}(\mathcal S^\star) = 0\,.
\]
(The case $K^\star= \infty$ is also possible!)
\end{conj}
\noindent This conjecture would imply that the limiting free energy of the ISP is given by 
\[
f(\be, h, \gamma) = \text{\sf{MP}}_{\be, h, \gamma}(\mathcal S^\star)\,,
\]
with $\mathcal S^\star$ the (unique) Parisi-M\'{e}zard structure associated to the parameters $(\be, h, \gamma)$. As explained in the introduction, this would also yield a solution of the ISP in case of infinite Erd\H{o}s-R\'{e}nyi random graphs, for {\it any} dilution-parameter $\gamma$. Indeed, it would hold that 
\[
\texttt{ISP}(\gamma) = \lim_{h \uparrow \infty} \lim_{\be \uparrow \infty} \frac{\text{\sf{MP}}_{\be, h, \gamma}(\mathcal S^\star)}{h}\,.
\]

For this, we should focus the attention on the $R_N$-term  in \eqref{restterm_hell}. A moment's thought suggests that,  {\it should} this term indeed vanish (for well chosen MP-structure), the following picture emerges: under the Gibbs measure, finitely many spins behave like a {\it mixture} of random variables! Slightly more precisely, it would follow that {\it given} a realization of the quenched magnetizations driven by the directing measure, spins are independent. Under the light of exchangeability and de Finetti-type theorems, see e.g. \cite{aldous}, such a result is perhaps not really surprising, half-jokingly: 
{\it if not mixtures, what else?} What is way less obvious is that the de Finetti measure driving the mixture should be ultrametric, i.e. hierarchically organized (a property which is inherited from the Derrida-Ruelle cascades). We have no convincing explanation for this: it is simply in line with the {\it  M\'{e}zard-Parisi Ansatz} \cite{giogio_mezard} for diluted models. For more on the role of exchangeability in spin glasses (mean field or diluted) with a particular focus on hierarchical structures, see 
\cite{panchy_four}.  

A rigorous approach to the M\'{e}zard-Parisi Ansatz, of which our conjecture is but one concrete case, is laid out in the works \cite{panchy, panchy_one, panchy_two, panchy_three, panchy_arxiv}. The approach is based on
many ingredients, such as perturbations of the ISP-Hamiltonian, the Ghirlanda-Guerra identities \cite{gg},  computations {\it \`{a} la  Aizenman-Sims-Starr} \cite{ass}, ultrametricity \cite{panchy}, exchangeability \cite{panchy_two}, Franz-Leone \cite{franz} upper-bounds, etc. It would take too long to explain any of this in detail, so we refer in particular to the introduction of \cite{panchy_arxiv} for an overview. 

We finally point out that the appeal of ultrametricity goes well beyond what may be perceived as some form of  aesthetic beauty. In fact, a {\it hands-on} approach to the issue, by this we mean a disorder-dependent construction of the "growing tree",  would have considerable impact on applications:  it would open the gate towards efficient {\it algorithms} for the construction of the maximal independent set, for given realization of the Erd\H{o}s-R\'{e}nyi random graph. The latter problem is naturally way more challenging than the mere (...) computation of the free energy. (The above conjecture, with the complexity lying underneath the surface, should   be seen as a cautionary note). Progress on this type of questions is yet nowhere in sight. 
To date, interpolations {\it \`{a} la Guerra} are the finest weapons available to address the low temperature behavior of spin glasses (be it diluted or mean field). In a wealth of models, these tools have proven tremendously effective for the computation of extensive quantities such as free energy, entropy, etc. Unfortunately, they also possibly change irreparably the models as far as the finer quantities are concerned.  \\

The remainder of the paper is devoted to the  proofs. 

\section{The free energy} \label{free_energy_proof}
In this section we give a  proof of Theorem \ref{free_energy} {\it assuming} Proposition \ref{prop_contraction} and Theorem \ref{full_ind}. As mentioned, the approach is based on integrating out one spin at a time (creating cavities), and exploiting the asymptotical decoupling. Some caution is needed, since the procedure of creating cavities induces small (but relevant) changes in the dilution-parameter: this is taken care by a telescopic decomposition. Precisely, denoting by $F_N(\be, h, \gamma)$ the {\it unnormalized} free energy, we write
\beq \label{tele} \bea
f_N(\be, h, \gamma) & = \frac{1}{N} \sum_{i=2}^N \left( F_i(\be, h, \gamma)  - F_{i-1}\left(\be, h, \frac{i-1}{i} \gamma \right) \right) \\
& \hspace{2.5cm} - \left( F_{i-1}(\be, h, \gamma) -  F_{i-1}\left(\be, h, \frac{i-1}{i} \gamma \right)  \right) + \frac{F_1(\be, h, \gamma)}{N}\,.
\eea \eeq
Recalling that $\gamma' = \frac{N-1}{N} \gamma$, we shorten
\beq \bea
& A_N \defi F_N(\be, h, \gamma)  - F_{N-1}\left(\be, h, \gamma' \right), \\
& B_N \defi F_{N-1}(\be, h, \gamma)  - F_{N-1}\left(\be, h, \gamma' \right).
\eea \eeq
With this notation, it follows from \eqref{tele} that
\beq
\lim_{N\to \infty} f_N(\be, \gamma, h) = \lim_{N\to \infty} A_N - \lim_{N \to \infty} B_N,
\eeq 
provided that both $A$- and $B$-limits exist. 

\subsection{The $A$-limit}
It holds
\beq \label{a_one}
A_N = \E \log \frac{\sum_{\s \in \Sigma_N} \exp H(\s)}{\sum_{\s \in \Sigma_{N-1}} \exp H_-(\s)}.
\eeq
We now proceed along the lines of \eqref{mean_one} and \eqref{mean_two}, i.e. we write $H(\s)$ in terms of 
$H_-(\s_1, \dots, \s_{N-1})$, and perform the trace over $\s_N\in \{0,1\}$. Equation \eqref{a_one} then reads 
\beq \label{a_two}
A_N = \E \log\left(1 + \left< \exp\left( h- \be \sum_{i\leq N-1} g_{i, N} \s_i \right)\right>_- \right)
\eeq
As the randomness in the Gibbs measure is independent of the $g_{i,N}$, and since the Gibbs measure is invariant (in distribution) under permutations of spins, we have in fact that
\beq \label{a_three}
A_N = \E \log\left(1 + \left< \exp\left( h- \be \sum_{i\leq S} \s_i \right)\right>_- \right),
\eeq
where $S = \sum\limits_{i<N} g_{i,N}$. 

We introduce 
\beq \label{ffunc}
f_k(x_1,..,x_k) \defi \exp\left( h- \be \sum_{i\leq k} x_i \right)\,,
\eeq
and 
\beq \label{gfunc}
g(x)\defi \log\left(1+x\right).
\eeq
With this notation, and integrating out $S$, \eqref{a_three} yields
\beq \bea \label{a_four}
A_N & = \sum\limits_{k=0}^{N-1}\PP(S=k)\E \log\left(1 + \left< \exp\left( h- \be \sum_{i\leq k} \s_i \right)\right>_- \right) \\
& =  \sum\limits_{k=0}^{N-1}\PP(S=k)\E g\left( \left<f_k(\s_1,..,\s_k)\right>_{-}\right).
\eea \eeq
Now denote by $X'=\left(X'_1,..,X'_k\right)$ a vector of independent $\nu_\star(\be,h,\gamma')$-distributed random variables, and conditionally on $X'$,  consider $B_1,..,B_k$  independent, Bernoulli($X'_i$) distributed random variables. We rewrite \eqref{a_four} as 
\beq \label{a_five}
A_N = \sum\limits_{k=0}^{N-1}\PP(S=k)\E g\left(\E\left[ f_k(B_1,..,B_k)|X'\right] \right)+r_N,
\eeq
where
\beq
r_N \defi \sum\limits_{k=0}^{N-1}\PP(S=k) \left(\E g\left( \left<f_k(\s_1,..,\s_k)\right>_{-}\right)- \E g\left(\E\left[ f_k(B_1,..,B_k)|X'\right] \right) \right),
\eeq
Applying the triangle inequality and Theorem \ref{full_ind} with parameters $N-1,\be,\gamma',h$ for each $k$ gives
\beq \bea \label{a_seven}
\left|r_N\right| & \leq \sum\limits_{k=0}^{N-1}\PP(S=k)\left| \E g\left(\E\left[ f_k(B_1,..,B_k)|X'\right] \right)-\E g\left(\left< f_k(B_1,..,B_k)\right>_-  \right)\right| \\
& \leq \sum\limits_{k=0}^{N-1}\PP(S=k) \alpha(\beta,\gamma') \|f_k\|_\infty L_g \frac{k^3}{N-1}.
\eea \eeq
We observe that all $f_k$ introduced in \eqref{ffunc} map to a subset of $\R^+$: restricted to this set, the function $g$ introduced in \eqref{gfunc} is $L_g$-Lipschitz with $L_g=1$. Furthermore, all $f_k$ are bounded by $1$, and $\alpha(\beta,\gamma') \leq \alpha(\beta,\gamma)$ since $\gamma'\leq \gamma$. This yields 
\beq \label{boundingS}
\left| r_N\right| \leq \frac{\alpha(\beta,\gamma)}{N-1}\sum\limits_{k=0}^{N-1}\PP(S=k)k^3= \frac{\alpha(\beta,\gamma)}{N-1} \E S^3 = o(1) \quad (N\to \infty)
\eeq
the last step since $S$ is a Binomial($N-1, \frac{\gamma}{N} $), in which case its third moment is bounded uniformly in $N$ (see Appendix). Using \eqref{boundingS} in \eqref{a_five}, and "undoing" the $S$-integration,
we therefore see that
$$ A_N = \E g\left(\E\left[ f_S(B_1,..,B_S)|X',S\right] \right)+o(1)\,.$$
Clearly $S$ converges weakly to a Poisson($\gamma$)-distributed random variable; furthermore, given $S$, $X'$ converges weakly to a vector of independent $\nu_\star(\be,h,\gamma)$-distributed random variables, by Proposition \ref{prop_contraction}-{\it ii)}. Remark that for any $k \in \N$ it holds that $0\leq f_k \leq e^h$; but the restriction of $g$ on $[0,e^h]$ is
bounded and continuous, so we may safely replace $S,X'$ by its weak limit $r,X$, at the price of a vanishing (in $N$) error, to wit: 
$$ A_N = \E g\left(\E\left[ f_r(B_1,..,B_r)|X,r\right] \right) + o(1)\,.$$
All in all, 
\beq
\lim\limits_{N\rightarrow \infty} A_N=\E \ln\left(1+\E\left[\exp\left(h-\beta \sum\limits_{i\leq r} B_i\right)|X,r\right] \right)
\eeq
\beq
 =\E \ln\left(1+e^h\prod\limits_{i\leq r}\E\left[\exp\left(-\beta B_i\right)|X\right] \right)
\eeq
\beq \label{a-lim}
=\E \ln\left(1+e^h\prod\limits_{i\leq r}\left[1-(1-e^{-\beta})X_i\right] \right). 
\eeq

\subsection{The $B$-limit} 
Recall that 
\[
B_N = F_{N-1}(\be, h, \gamma) - F_{N-1}(\be, h, \gamma'), \quad \text{where}\quad \gamma' = \frac{N-1}{N} \gamma.
\]
We are thus comparing two systems defined on the same configuration space $\Sigma_{N-1}$, but with slightly different dilution-parameters. This can be taken into account by a coupling procedure, i.e. introducing fresh random variables
$$(\hat g_{ij})_{1\leq i<j\leq N-1} \;    \text{independent Bernoulli with success probability}\;  \frac{\gamma}{(N-1)(N-\gamma)}\,,$$ 
independent of the $g's$, in which case, using that  the $ (g_{ij})_{1\leq i<j\leq N-1}$ are independent Bernoulli($\gamma'/\left(N-1\right) $), it is immediate to check that  
$$\left(g_{ij} + \boldsymbol{1}\{g_{ij} = 0\} \hat g_{ij}\right)_{1\leq i<j\leq N-1} \; \text{are independent Bernoulli}\left(\frac{\gamma}{ N-1} \right).$$
The following representation therefore arises: 
\[
B_N = \E \log \frac{\sum_{\s \in \Sigma_{N-1}} \exp\left( h \sum_i \s_i - \be \sum_{i<j} \left[ g_{ij} + \boldsymbol{1}\{g_{ij} = 0\} \hat g_{ij} \right]  \s_i \s_j\right)  }{ \sum_{\s \in \Sigma_{N-1}} \exp\left(h \sum_i \s_i  - \be \sum_{i<j} g_{ij} \s_i \s_j \right) },
\]
or, which is the same, 
\beq \label{reform_b}
B_N = \E \log \left< \exp\left( -  \be \sum_{i<j}  \boldsymbol{1}\{g_{ij} = 0\} \hat g_{ij}  \s_i \s_j \right) \right>_- \,.
\eeq
We now write 
\beq \bea
& \exp\left( -  \be \sum_{i<j}  \boldsymbol{1}\{g_{ij} = 0\} \hat g_{ij}  \s_i \s_j \right)  \\
& = \exp\left( -  \be \sum_{i<j} \hat g_{ij}  \s_i \s_j \right) \exp\left( \be \sum_{i<j}  \left(1-\boldsymbol{1}\{g_{ij} = 0\}\right) \hat g_{ij}  \s_i \s_j \right)
\\
& = \exp\left( -  \be \sum_{i<j} \hat g_{ij}  \s_i \s_j \right)
 r_N,\quad \text{say}.
\eea \eeq
It is easily seen that 
\beq \label{r_N}
\mid \log r_N \mid \leq \be \sharp\left\{ (i,j): \, 1\leq i < j \leq N-1,\; \text{and}\; g_{ij} = \hat g_{ij} = 1 \right\}.
\eeq
The r.h.s. of \eqref{r_N} is $\beta$ times a Binomial random variable of mean 
$\frac{\gamma^2 \left(N-2\right)}{2N\left(N-\gamma\right)}$, and this in turns implies  that
\beq \label{almost_done}
B_N = \E \log \left< \exp\left( -  \be \sum_{i<j} \hat g_{ij}  \s_i \s_j \right) \right>_-  + o(1)\,\quad (N\to \infty).
\eeq
We now claim that 
\beq \bea \label{claim}
\text{"spins appearing in the exponential on the}\\ 
\text{r.h.s. of \eqref{almost_done} can do it only once."}
\eea \eeq 
Precisely, we consider the event 
\beq \label{only_once}
\Omega_N \defi \bigcap\limits_{1 \leq i < j \leq N-1} \Big\{ \hat g_{ij} = 1 \Rightarrow 
 \hat g_{ik}=0 \; \forall_{k\leq N-1, k \neq j}\; \text{and}\; \hat g_{kj}=0 \; \forall_{k\leq N-1, k \neq i} 
\Big\} \,.
\eeq 
An upper bound for the total contribution of the complement $\Omega_N^c$ to $B_N$ is obtained by setting $\s_i \s_j = 1$ for all pairs $(i,j)$; this steadily yields the estimate
\beq  \label{noc}
\left| \E\left[ \boldsymbol{1}_{\Omega^c_N} \log \left< \exp\left( -  \be \sum_{i<j} \hat g_{ij}  \s_i \s_j \right) \right>_Y  \right] \right| \leq  \be \E\left[ \boldsymbol{1}_{\Omega^c_N}\sum_{i<j} \hat g_{ij} \right] \,.
\eeq
Estimating the indicator by the number of pairs of $\hat g$ that do not satisfy its condition gives
\beq \label{missing_bounds}
\boldsymbol{1}_{\Omega^c_N} \leq \sum\limits_{1\leq i<j<j'<N} \hat{g}_{ij} \hat{g}_{ij'} + \sum\limits_{1\leq i<i'<j'<N} \hat{g}_{ij} \hat{g}_{i'j}\,.
\eeq
Applying this estimate using that the $\hat g's$ are independent Bernoulli with success probability $\frac{\gamma}{(N-1)(N-\gamma)}$, a simple computation shows that 
\beq 
\E\left[ \boldsymbol{1}_{\Omega^c_N}\sum_{i<j} \hat g_{ij}\right] \leq N^3 \left(\frac{\gamma}{\left(N-1\right)\left(N-\gamma\right)}\right)^2+N^5\left(\frac{\gamma}{\left(N-1\right)\left(N-\gamma\right)}\right)^3 \,,
\eeq
which is indeed vanishing in the limit $N\to \infty$: this proves (and formalizes) claim \eqref{claim}. 

But on $\Omega_N$, all spins appearing in the exponential of (\ref{almost_done}) are different: since the Gibbs measure is independent of the event $\Omega_N$ (and invariant in distribution under spin-permutation), setting $S \defi \sum\limits_{i<j<N} \hat{g}_{ij}$, we get
\beq
B_N = \E \boldsymbol{1}_{\Omega_N}\log \left< \exp\left( -  \be \sum\limits_{i=1}^S \s_{2i-1} \s_{2i} \right) \right>_-  + o(1)\,\quad (N\to \infty).
\eeq
Integrating out $S$, we thus obtain
\beq
B_N = \sum\limits_{k=0}^{(N-1)/2}\PP(S=k,\Omega_N) \E \log \left< \exp\left( -  \be \sum\limits_{i=1}^k \s_{2i-1} \s_{2i} \right) \right>_-  + o(1)\,.
\eeq
(The above sum runs to $(N-1)/2$ only because for bigger $k$ it plainly holds that $\{S=k\} \cap \Omega_N = \emptyset$). Using Theorem \ref{full_ind} along the lines of \eqref{a_three}-\eqref{a_seven}, but in this case with 
\[
f_k(B) \defi  \exp\left( -  \be \sum\limits_{i=1}^k B_{2i-1} B_{2i} \right),
\]
and $g_k$ the natural logarithm restricted to $\left[ \min f_k, \max f_k\right]$, we obtain
\beq\label{b_onlydetailsleft}
B_N = \E \boldsymbol{1}_{\Omega_N} \log \E\left[ \exp\left( -  \be \sum\limits_{i=1}^S B_{2i-1} B_{2i} \right)|X',S \right]  + Q_N, 
\eeq
where 
\[
Q_N \defi \frac{\alpha(\beta,\gamma')}{N-1} \E\left[ S^3\|f_S\|_\infty L_g \right] + o(1) \qquad (N\to \infty)\,.
\]
Observe that $f_S(B) \in [e^{-\beta S}, 1]$; on this set (the restriction of) $g_S$ is $L_{g_S}$ -Lipschitz where $L_{g_S} \defi e^{\beta S}$. This implies 
\beq \label{bla}
Q_N \leq 
\frac{\alpha(\beta,\gamma)}{N-1} \E\left[ S^3 e^{\beta S}\right]+o(1)\,.
\eeq
It is not difficult to check (see the Appendix) that $\E\left[ S^3 e^{\beta S}\right]$ is uniformly bounded in $N$,  
hence \eqref{bla} vanishes in the large $N$-limit. 

Concerning the first term on the r.h.s. of \eqref{b_onlydetailsleft}: analogous arguments as those used to introduce the $\boldsymbol 1_{\Omega_N}$-restriction (see in particular \eqref{claim} and ff.) steadily
yield
\beq
\E (1-\boldsymbol{1}_{\Omega_N}) \log \E\left[ \exp\left( -  \be \sum\limits_{i=1}^S B_{2i-1} B_{2i} \right)|X',S \right] = o(1).
\eeq
All in all,  
\beq
B_N = \E \log \E\left[ \exp\left( -  \be \sum\limits_{i=1}^S B_{2i-1} B_{2i} \right)|X',S \right] +o(1).
\eeq
Using the fact that given $X$ the $B_i$ are independent we have
\beq
B_N = \E \sum\limits_{i=1}^S\log \E\left[\exp\left( - \be    B_{2i-1} B_{2i} \right) |X',S\right] +o(1).
\eeq
Computing the conditional expectation gives 
\beq
B_N = \E \sum\limits_{i=1}^S\log \left[1-(1-e^{-\beta})X'_{2i-1}X'_{2i} \right]  +o(1).
\eeq
and since all involved random variables are independent, 
\beq \bea
B_N &= \E[S] \E \log \left[1-(1-e^{-\beta})X'_{1}X'_{2} \right]  +o(1) \\
& = \frac{\gamma}{2} \E \log \left[1-(1-e^{-\beta})X'_{1}X'_{2} \right] + o(1).
\eea \eeq
Since $(X'_1,X'_2)$ converges weakly to $\nu_\star(\be,h,\gamma)^{\otimes 2}$, it steadily follows from Proposition \ref{prop_contraction} that
\beq \label{b-lim}
\lim_{N\to \infty} B_N = \frac{\gamma}{2} \E\log\left( 1-\left(1-e^{-\be} \right) X_1 X_2 \right)\,,
\eeq
where $X_1, X_2$ are independent random variables, distributed according to the fixpoint-solution of $\nu = T_{\be, h, \gamma}\nu$. Theorem \ref{free_energy} thus follows from \eqref{a-lim} and \eqref{b-lim}. \\
${}$ \hfill $\square$

\section{The $T$-operator, and continuity of the fixpoints}\label{proof_T}
We present here a proof of Proposition \ref{prop_contraction}. We first show that the $T$-operator is, in the replica symmetry phase, a contraction. 
\begin{proof}[Proof of Proposition \ref{prop_contraction}-i)]
For  $\mu,\nu$ probability measures on $[0,1]$, we claim that 
\beq \label{contrclaim}
d\left( T_{\be, h, \gamma}(\mu), T_{\be, h, \gamma}(\nu) \right) \leq C(\be, \gamma) d(\mu, \nu)\,,
\eeq
where $C(\be, \gamma)$ is given by \eqref{replicasymmetry}: this will naturally imply Proposition \ref{prop_contraction}-{\it i)}. \\

\noindent To see \eqref{contrclaim} we first observe that, by definition of $T_{\be, h, \gamma}$, it holds
\beq \bea \label{derivestimate} 
& d\left( T_{\be, h, \gamma}(\mu), T_{\be, h, \gamma}(\nu) \right) \\
& \qquad \qquad \leq \inf \E\left|\frac{1}{1+ \left< \exp\left(h-\beta\sum_{i=1}^r \sigma_i\right) \right>^{-1}_X}-\frac{1}{1+ \left< \exp\left(h-\beta\sum_{i=1}^r \sigma_i\right) \right>^{-1}_Y} \right|, 
\eea \eeq
where $X$ is a sequence of independent $\mu$-distributed random variables, $Y$ is a sequence of independent $\nu$-distributed random variables, the infimum is over all couplings of $X$ and $Y$, and $r$ is Poisson($\gamma$)-distributed random variable which is independent of $(X,Y)$. \\
Introduce now the function $m: [0,1]^r \to [0,1]$,
\[
x\mapsto m(x) \defi \left(1+ \left< \exp\left(h-\beta\sum_{i=1}^r \sigma_i\right) \right>^{-1}_x\right)^{-1}\,.
\] 
One easily checks that
\[
\| \partial_{x_i} m\|_{\infty} \leq \frac{e^\beta -1}{4}\,, 
\]
if $ i\leq r$, and zero otherwise. It therefore follows that  
\[
\eqref{derivestimate} \leq \inf \E\ \sum\limits_{i=1}^\infty \| \partial_{x_i} m\|_{\infty} \left| X_i-Y_i \right|\leq \frac{e^\beta -1}{4}\inf \E \sum\limits_{i=1}^r \left| X_i-Y_i \right|.
\]
We now upper-bound the r.h.s. above by restricting the infimum to couplings for which $(X_i,Y_i)_{i\in \N}$ is an i.i.d. sequence : since  the Poisson$(\gamma)$ is independent of everything else, we have
\[
\eqref{derivestimate} \leq \frac{e^\beta -1}{4} \E[r] \inf \E \left| X_1-Y_1 \right|=  \frac{e^\beta -1}{4} \E[r] d\left(\mu,\nu\right) =  \frac{e^\beta -1}{4} \gamma d\left(\mu,\nu\right) \,.
\]
It is immediate to check that $$\frac{e^\beta -1}{4}\gamma \leq C(\gamma,\beta)< 1,$$
hence claim \eqref{contrclaim} is proven, and Proposition \ref{prop_contraction}-i) follows. 
\end{proof}

We next prove the continuity estimates on the solution of the fixpoint-equations. 

\begin{proof}[Proof of Proposition \ref{prop_contraction}-ii)]  Without loss of generality, we assume that 
$\gamma \leq \gamma'$, in which case one immediately checks that $(1-C(\be, \gamma))^{-1} \leq (1-C(\be, \gamma'))^{-1}$. Shorten $\nu_\star \defi \nu_\star\left(\be,h,\gamma \right)$, and $\nu'_\star \defi \nu_\star\left(\be,h,\gamma' \right)$, as well as $T\defi T_{\be,h,\gamma}$, and $T' \defi T_{\be,h,\gamma'}$.  Since $\nu_\star$ and $\nu'_\star$ are the (unique) fixpoints of the corresponding operators, 
\beq\label{contrcalc} 
d\left(\nu_\star,\nu'_\star\right)= d\left(T\nu_\star,T'\nu'_\star\right)\leq d\left(T\nu_\star,T\nu'_\star\right)+d\left(T\nu'_\star,T'\nu'_\star\right)\,,
\eeq
the second step by the triangle inequality. We now apply Proposition \ref{prop_contraction}-{\it i)} to the first term of 
 \eqref{contrcalc}  to get 
\[
d\left(\nu_\star,\nu'_\star\right) \leq C(\gamma,\beta) d\left(\nu_\star,\nu'_\star\right)+d\left(T\nu'_\star,T'\nu'_\star\right)\,,
\]
or, which is the same, 
\beq \label{cheballe}
d\left(\nu_\star,\nu'_\star\right) \leq \frac{1}{1-C(\gamma,\beta)}d\left(T\nu'_\star,T'\nu'_\star\right).
\eeq
We now focus on the r.h.s of \eqref{cheballe}: since $T\nu'_\star$ and $T'\nu'_\star$ are both probability measures on $[0,1]$, for {\it any} event $\tilde \Omega \subset \Omega$, and with  $Z$ and $Z'$ random variables distributed according to $T\nu'_\star$ and, respectively, $T'\nu'_\star$, it holds: 
\beq \label{cheballe_due}
d\left(T\nu'_\star,T'\nu'_\star\right) \leq \E \textbf{1}_{\tilde \Omega}\left|Z-Z'\right|+ \PP\left(\tilde \Omega^c\right)\,.
\eeq
As for a concrete choice of the random variables appearing in \eqref{cheballe_due}, we proceed as follows: 
we let $r$ to be Poisson($\gamma$)-distributed, $r'$ Poisson($\gamma'$)-distributed, and the sequence $X$ consists of independent $\nu'_\star$-distributed random variables. By definition of the $T$-operator, we may choose $Z,Z'$ as follows:
\[ \bea
Z &\defi \left(1+ \left< \exp\left(h-\beta\sum_{i=1}^r \sigma_i\right) \right>^{-1}_X \right)^{-1}, \\
Z' &\defi \left(1+ \left< \exp\left(h-\beta\sum_{i=1}^{r'} \sigma_i\right) \right>^{-1}_X \right)^{-1}.
\eea \]
We now come to a specific choice of the $\tilde \Omega$-event, to wit: $$\tilde \Omega \defi \{\omega \in \Omega: r(\omega) = r'(\omega) \}.$$ Picking a coupling of $r,r'$ which maximizes $\PP(r=r')$ yields  
$$\PP\left[ \tilde \Omega^c \right] \leq \frac{1}{2} d_{TV}(r, r'), $$
and since $Z = Z'$ on  $\tilde \Omega$ , \eqref{cheballe_due} becomes 
\beq \label{cheballe_tre}
d\left(T\nu'_\star,T'\nu'_\star\right)  \leq   \PP\left(\tilde \Omega^c\right)\leq \frac{1}{2}d_{TV}(r,r') \leq  \left|\gamma-\gamma' \right|,
\eeq
the last inequality by well-known estimates on the total variation distance of two Poisson distributions. Plugging \eqref{cheballe_tre} in \eqref{cheballe} settles the claim of Proposition\ref{prop_contraction}-{\it ii)}. 
\end{proof}

\section{Asymptotical decoupling}\label{proof}

The proof of Theorem \ref{full_ind} is done in two steps, In a first step, Section \ref{proof_spin_ind} below, we  prove a {\it quenched decoupling}: the quenched Gibbs measure can be replaced by a random product measure with identical marginals (remark that the latter is uniquely characterized by the magnetizations). In a second step we will prove the {\it annealed decoupling}, namely that the magnetizations, under the $g$-disorder, are approximately independent, and $\nu_\star$-distributed: this will be done in Section \ref{proof_mag_ind}. We will then show in Section \ref{finish_full_ind} how to combine quenched and annealed decoupling to derive Theorem \ref{full_ind}.

\subsection{Quenched decoupling}\label{proof_spin_ind}
In this section we prove that the Gibbs measure restricted to finitely many spins approaches, for large $N$, a (quenched) product measure. This is encoded in the following (at first sight presumably opaque) statement. 

\begin{prop}\label{spin_ind}
Let $C_1,C_2,h,\beta,\gamma\geq 0$ and suppose $C(\be, \gamma)<1$ holds. Then for all $f:\Sigma_k\rightarrow \R$,$f':\Sigma_k\rightarrow \R^+$ with the properties
\begin{itemize}
\item $\left|f\right|\leq f'$ 
\item $|f(x)-f(\hat{x})|\leq C_1 f'(x)$ for all $x,\hat{x}\in \Sigma_k$ where $x$ and $\hat{x}$ are different in one coordinate  
\item $|f'(x)-f'(\hat{x})|\leq C_2 f'(x)$ for all $x,\hat{x}\in \Sigma_k$ where $x$ and $\hat{x}$ are different in one coordinate  
\end{itemize}
the following holds:
\beq\label{spin_ind_estim} \E \left| \frac{\left< f\left(\sigma_1,..,\sigma_k\right)\right>}{\left< f'\left(\sigma_1,..,\sigma_k \right)\right>}-\frac{\left< f\left(\sigma_1,..,\sigma_k\right)\right>_Y}{\left< f'\left(\sigma_1,..,\sigma_k\right)\right>_Y}\right|\leq k\left(C_1+C_2\right) \frac{kB+B^*}{N},\eeq
where $Y = \left(\left<\sigma_i\right>\right)_{i\leq N} $ and $B,B^*$ are increasing in $\beta$,$\gamma$ and are given by 
$$ B = \frac{\gamma}{1- C\left(\gamma,\beta\right)} \quad\mbox{ and }\quad B^* = \frac{\frac{1}{2}\gamma^2 e^{2\beta}C\left(\gamma,\beta\right)}{1-C\left(\gamma,\beta\right)}.$$
\end{prop}
To see that the above Proposition indeed implies the approximate decoupling of the (quenched) Gibbs measure, consider the following situation: let $f$ be a bounded (non-zero) function, and set $f' \defi \|f\|_\infty$.  By Proposition \ref{spin_ind} with $C_1= 2$ and $C_2= 0$ it follows that
\beq
\E \left|\left< f\left(\sigma_1,..,\sigma_k\right)\right>-\left< f\left(\sigma_1,..,\sigma_k\right)\right>_Y\right|\leq 2k \|f\|_\infty \frac{kB+B^*}{N}.
\eeq
The error when approximating the quenched Gibbs measure with a product measure is therefore vanishing. \\

Although the proof of Proposition \ref{spin_ind} relies on the simple idea of decoupling on spin at a time, the 
rigorous implementation is quite involved. Controlling the error generated by a single step of the procedure is the content of the following Lemma. Before that, we need to introduce some additional notation which captures 
the concept of "partially decoupled" (the meaning of which will become clear in the course of the proof). 

This is achieved by considering {\it replicas} $\s^0,\s^1,..,\s^n \in \Sigma_N$, namely configurations 
which are independently drawn from $\mathcal{G}_{N,\beta,\gamma,h}$. More precisely, given $\mathcal{G}_{N,\beta,\gamma,h}$, the vector $(\s^0,\s^1,..,\s^n)\in \Sigma_N^n$ is $\mathcal{G}^{\otimes n+1}_{N,\beta,\gamma,h}$-distributed. To lighten notations, we stick to the convention of omitting the underlying parameters, i.e. we write $\left<\right>^{\otimes n}$ for (quenched) expectation w.r.t. 
$\mathcal{G}^{\otimes n}_{N,\beta,\gamma,h}$ on the $N$-system, and $\left<\right>^{\otimes n}_{-}$ for expectation w.r.t. $\mathcal{G}^{\otimes n}_{N-1,\beta,\gamma', h}$ on the $(N-1)$-system, where $\gamma' = \frac{N-1}{N}\gamma$ is the reduced dilution-parameter. Remark that, by these very definitions,  the following identities hold true: 
\[
\left<f\left(\sigma^0,\sigma^1,..,\sigma^n\right)\right>^{\otimes n} = \sum\limits_{\s^1,..,\s^n\in \Sigma_N} f\left(\sigma^0,\sigma^1,..,\sigma^n\right) \prod\limits_{l=0}^{n} \mathcal{G}\left(\s^l\right), 
\]
and
\[
\left<f(\sigma_1^1,\sigma_2^1..,\sigma_N^1)\right>^{\otimes n} =\left<f(\sigma_1,\sigma_2..,\sigma_N)\right> \,.
\] 
Furthermore, for $Y = \left(\left<\s_i\right>\right)_{i\leq N}$, it holds that 
\[
\left<f(\sigma_1^1,\sigma_2^2..,\sigma_N^N)\right>^{\otimes n} =\left<f(\sigma_1,\sigma_2..,\sigma_N)\right>_Y.
\]
Finally, since we will consider functions that depend only on a fixed number of sites, it is convenient to denote the {\it rows} of the matrix  $\left(\sigma_i^l\right)_{1\leq i\leq N,1\leq l\leq n}$ by bold $\boldsymbol{\s}_i= \left(\sigma_i^1,\sigma_i^2..,\sigma_i^n\right)$. (Not to be confused with the {\it columns} of the matrix, which stand for the replicas $\s^1$ to $\s^n$). \\

In order to prove Proposition \ref{spin_ind}, the following Lemma is needed: 

\begin{lem}\label{swap_one}
Assume $\be,\gamma$ are such that $C(\be,\gamma)<1$. Let $m \leq k \leq n \leq N \in\N$. For functions $f:\Sigma_k^n \rightarrow \R$, $f':\Sigma_k^n \rightarrow \R^+$ with the property, that $|f(x)-f(\hat x)| \leq f'(x)$ for all $x,\hat x$ that are only different in one of the $k n$ entries, the following holds:

\beq\label{exchange} \E\left|\frac{\left<f \left(\boldsymbol{\s}_1,..,\boldsymbol{\s}_m,\boldsymbol{\s}^{\leftrightarrow}_{m+1},..,\boldsymbol{\s}^{\leftrightarrow}_{k}\right)-f\left(\boldsymbol{\s}_1,..,\boldsymbol{\s}_{m-1},\boldsymbol{\s}^{\leftrightarrow}_{m},..,\boldsymbol{\s}^{\leftrightarrow}_{k} \right)\right>^{\otimes n}}{\left<f'\left(\boldsymbol{\s}_1,..,\boldsymbol{\s}_m,\boldsymbol{\s}^{\leftrightarrow}_{m+1},..,\boldsymbol{\s}^{\leftrightarrow}_{k} \right)\right>^{\otimes n}} \right| \leq \frac{kB+B^*}{N},\eeq
where $B,B^*$ are the constants from proposition \ref{spin_ind} and $\boldsymbol{\s}^{\leftrightarrow}_i =\left(\sigma_i^0,\sigma_i^2..,\sigma_i^n\right)$.
\end{lem}

\begin{proof}
We first clarify the relation between the Hamiltonians $H$ and $H_{-}$, and between the Gibbs measures $\mathcal G$ and $\mathcal G_{-}$: plainly, 
\beq\label{relham} H(\sigma) = H_{-}\left(\sigma\right) + \sigma_N \left(h-\beta\sum\limits_{i<N} g_{i,N}\sigma_i \right), \eeq
by a slight abuse of notation ($H_{-}$ does not depend on $\sigma_N$).It follows from \eqref{relham} that for any function $f$ of $\left(\sigma^0,..,\sigma^n\right)$
$$ \left<f\right>^{\otimes n} =  \frac{1}{Z}\sum\limits_{\sigma^0,..,\sigma^n\in\Sigma^{n+1}_{N-1}}\left(\sum\limits_{\sigma^0_N,..,\sigma^n_N\in\{0,1\}} f \mathcal{E}\prod\limits_{0\leq l\leq n}\exp H_{-}\left(\sigma^l \right)\right) = \frac{Z_{-}}{Z}\left<\sum\limits_{\sigma^0_N,..,\sigma^n_N\in\{0,1\}} f \mathcal{E}\right>_{-}^{\otimes n},$$
where $\mathcal{E} = \mathcal{E}_{h,\beta,\gamma,N}\left(\sigma^0,..,\sigma^n \right) =  \exp\left(\sum\limits_{0\leq l\leq n}\sigma^l_N \left(h-\beta\sum\limits_{i<N} g_{i,N}\sigma_i^l\right) \right)$. Considering the fraction of two such expectations gives a self similar link between the $N$-system and the $N-1$-system
\beq\label{selfsim2}\frac{\left<f\right>^{\otimes n}}{\left<f'\right>^{\otimes n}} = \frac{\left<\Av f \mathcal{E}\right>_{-}^{\otimes n}}{\left<\Av f' \mathcal{E}\right>_{-}^{\otimes n}}, \eeq
where $\Av$ denotes the average over all $\sigma_N^i\in \{0,1\}$.

We now proceed to prove the claim: this is done by induction on $N$, i.e. propagating the estimate from the $N-1$- to the $N$-system, thereby using equation (\ref{selfsim2}).

For $N=1$ and any choice of $\gamma,B,B^*\geq 0$ we have $k\leq 1$ therefore $f$ only depends on one coordinate at most. $f$ is constant for $k=0$, whereas for $k=1$, by symmetry, $\left<f \left(\boldsymbol{\s}_1\right)-f\left(\boldsymbol{\s}^{\leftrightarrow}_1\right)\right> = 0$. It follows that the numerator of the left hand side of (\ref{exchange}) is zero. This proves the lemma for $N=1$. Let $N \geq 2$, assume $C(\be,\gamma)<1$ and that the lemma holds for $N-1$ and all $\gamma'\leq \gamma$. Let $f:\Sigma_k^n \rightarrow \R$, $f':\Sigma_k^n \rightarrow \R^+$ be functions with the property, that $|f(x)-f(\hat x)| \leq f'(x)$ for all $x,\hat x$ that are only different in one of the $k n$ entries. We set $f_1= f \left(\boldsymbol{\s}_1,..,\boldsymbol{\s}_m,\boldsymbol{\s}^{\leftrightarrow}_{m},..,\boldsymbol{\s}^{\leftrightarrow}_{k-1},\boldsymbol{\s}_{N}\right)$, $f_2 = f\left(\boldsymbol{\s}_1,..,\boldsymbol{\s}_m,\boldsymbol{\s}^{\leftrightarrow}_{m},..,\boldsymbol{\s}^{\leftrightarrow}_{k-1},\boldsymbol{\s}^{\leftrightarrow}_{N}\right)$ and $f'= f'\left(\boldsymbol{\s}_1,..,\boldsymbol{\s}_m,\boldsymbol{\s}^{\leftrightarrow}_{m},..,\boldsymbol{\s}^{\leftrightarrow}_{k-1},\boldsymbol{\s}_{N}\right)$. As the the replicated Gibbs measure is invariant in distribution among swapping of sites we have 
\[
\mbox{l.h.s.} (\ref{exchange}) = \E\left|\frac{\left<f_1-f_2 \right>^{\otimes n}}{\left<f' \right>^{\otimes n}}\right|.
\]
Conditioning on $(g_{i,N})_{i< N}$, for any  event $\tilde  \Omega \subset \Omega$ which is measurable with respect to $\left(g_{i,N}\right)_{i<N}$, it holds: 
\beq\label{omegadecomp}\text{l.h.s. of } (\ref{exchange}) \leq \E\textbf{1}_{\tilde \Omega}\E\left[\left|\frac{\left<f_1-f_2 \right>^{\otimes n}}{\left<f'\right>^{\otimes n}}\right| |\left(g_{i,N}\right)_{i<N} \right] + \PP(\tilde  \Omega^c)\,.
\eeq
(Remark that the fraction is bounded by one). Using (\ref{selfsim2}) and writing $\E_{N}$ for this conditional expectation 
\beq\label{selfsimexpec}\E_{N}\left|\frac{\left<f_1-f_2 \right>^{\otimes n}}{\left<f' \right>^{\otimes n}}\right| = \E_{N}\left|\frac{\left<\Av \left(f_1-f_2\right) \mathcal{E}\right>_{-}^{\otimes n}}{\left<\Av f' \mathcal{E}\right>_{-}^{\otimes n}}\right|. 
\eeq
The next step amounts to decomposing the expression $\Av  \left(f_1-f_2\right)\mathcal{E}$ in such 
a way that the induction hypothesis kicks in. To do this we first introduce
\[
\tilde \Omega \defi \{g_{i,N} = 0 \mbox{ for } i=1,..,k-1\}, \quad \text{and} \quad J \defi \{i<N|g_{i,N}\neq 0\}\,.
\]
Observe that, on $\tilde \Omega$, the function $f_1-f_2$ does {\it not} depend on the coordinates which appear in $J$, whereas $\mathcal{E}$ depends {\it solely} on these, and the $N$-coordinate: this "separation" will naturally turn out to be very useful. Writing $T_i$ for the operator that swaps $\sigma_i^0$ and $\sigma_i^1$ we have
$$\left(\Av \left(f_1-f_2 \right) \mathcal{E}\right)\circ \prod\limits_{i\in J} T_i = \Av \left[\left(f_1-f_2 \right) \left(\mathcal{E} \circ \prod\limits_{i\in J} T_i\right)\right].$$

Remark that $\Av$ is invariant under swapping of $\sigma_N^0$ and $\sigma_N^1$, yet the procedure 
turns $f_1$ into $f_2$, in particular it holds that $\left(f_1-f_2\right)\circ T_N = f_2-f_1$ and therefore 
$$ \Av \left(f_1-f_2 \right) \mathcal{E} = - \left(\Av \left(f_1-f_2 \right) \mathcal{E}\right)\circ \prod\limits_{i\in J} T_i,$$
since $\mathcal{E}\circ \prod\limits_{i\in J\cup \{N\}}  T_i = \mathcal{E}$. Decomposing telescopically by swapping one spin at a time gives  
\beq\label{decompfs} \Av \left(f_1-f_2 \right) \mathcal{E} = \frac{1}{2}(\Av \left(f_1-f_2 \right) \mathcal{E}- (\Av \left(f_1-f_2 \right) \mathcal{E})\circ \prod\limits_{i\in J} T_i) =  \frac{1}{2}\sum\limits_{s=1}^{|J|} f_s-f_s\circ T_{i_s},\eeq 
where  $f_s = (\Av \left(f_1-f_2 \right) \mathcal{E})\circ \prod\limits_{1\leq s' < s} T_{i_s'}$ with the convention that $J = \{i_1,...,i_{|J|}\}$. 
Applying this decomposition to (\ref{selfsimexpec}), by the triangle inequality
$$\textbf{1}_{\tilde  \Omega}\E_{N}\left|\frac{\left<f_1-f_2 \right>^{\otimes n}}{\left<f' \right>^{\otimes n}}\right| \leq \textbf{1}_{\tilde \Omega} \frac{1}{2}\sum\limits_{s=1}^{|J|}\E_{N}\left|\frac{\left< f_s-f_s\circ T_{i_s}\right>_{-}^{\otimes n}}{\left<\Av f' \mathcal{E}\right>_{-}^{\otimes n}}\right|.$$
Plugging this into \eqref{omegadecomp}, we therefore obtain the estimate
\beq\label{estimateviadecomp}\text{l.h.s. of } (\ref{exchange}) \leq \E\textbf{1}_{\tilde \Omega} \frac{1}{2}\sum\limits_{s=1}^{|J|}\E_{N}\left|\frac{\left< f_s-f_s\circ T_{i_s}\right>_{-}^{\otimes n}}{\left<\Av f' \mathcal{E}\right>_{-}^{\otimes n}}\right| + \PP(\tilde \Omega^c)\eeq
The only missing ingredient is to find $C>0$ such that $| f_s-f_s\circ T_{i_s}|\leq C \Av f' \mathcal{E}$, for then we could apply the induction assumption to get 
\beq\label{whatCdoes} C \E_{N}\left|\frac{\left< f_s-f_s\circ T_{i_s}\right>_{-}^{\otimes n}}{\left<C \Av f' \mathcal{E}\right>_{-}^{\otimes n}}\right| \leq C \frac{\left(k+|J|\right)B+B^*}{N-1} \,.
\eeq
To identify such a $C$, we set $\mathcal{E}^l \defi \exp(\sigma^l_N (h-\beta\sum\limits_{i<N} g_{i,N}\sigma_i^l) )$ in which case $\mathcal{E}= \prod\limits_{0\leq l\leq n}\mathcal{E}^l$, and analyse the construction of $f_s$ on $\tilde{\Omega}$. By definition
$$\left|f_s-f_s\circ T_{i_s} \right| = \left|(\Av \left(f_1-f_2 \right) \mathcal{E})\circ \prod\limits_{1\leq s' < s} T_{i_s'}-(\Av \left(f_1-f_2 \right) \mathcal{E})\circ \prod\limits_{1\leq s' \leq s} T_{i_s'}\right|.$$
Separating the terms that depend on the permutation $T_{i_s}$, the above equals, on $\tilde  \Omega$,
\beq \bea
& \left|\Av \left(f_1-f_2 \right) \prod\limits_{2\leq l\leq n} \mathcal{E}^l \left( \prod\limits_{l=0,1}\exp\left(\sigma^l_N \left(h-\beta\sum\limits_{1\leq s'\leq |J|,s'\neq s} \sigma_{i_{s'}}^l \right) \right)\circ \prod\limits_{1\leq s' < s} T_{i_{s'}} \right) \right.\cdot \\
& \hspace{4.5cm} \cdot \left.\left( \prod\limits_{l=0,1}e^{-\sigma^l_N \beta \sigma_{i_{s}}^l} -\prod\limits_{l=0,1}e^{-\sigma^l_N \beta \sigma_{i_{s}}^l}\circ  T_{i_{s}}\right)\right|.
\eea\eeq
This is bounded above by
$$ \left|\Av \left(f_1-f_2 \right)  \mathcal{E}\right| \sup\limits_{\sigma^1,\sigma^2\in \Sigma_N}\left(\frac{\prod\limits_{l=0,1}\exp\left(\sigma^l_N \left(h-\beta\sum\limits_{1\leq s'\leq |J|,s'\neq s} \sigma_{i_{s'}}^l \right) \right)\circ \prod\limits_{1\leq s' < s} T_{i_{s'}}}{\mathcal{E}^0 \mathcal{E}^1}\right) \left(1-e^{-2\beta}\right).$$
By  definition of $\mathcal{E}^l$, the $\sup$  is equal to
\[
\sup\limits_{\sigma^1,\sigma^2\in \Sigma_N} \prod\limits_{l=0,1}\exp\left(\sigma^l_N \left(\left(h-\beta\sum\limits_{1\leq s'\leq |J|,s'\neq s} \sigma_{i_{s'}}^l \right)\circ \prod\limits_{1\leq s' < s} T_{i_{s'}} - \left(h-\beta\sum\limits_{1\leq s'\leq |J|} \sigma_{i_s'}^l\right)\right) \right), 
\]
hence $h$ and all $\sigma_{i_{s'}}$ for $s'>s$ cancel, and the above equals
\[ \bea
\sup\limits_{\sigma^1,\sigma^2\in \Sigma_N}\prod\limits_{l=0,1}\exp\left(\sigma^l_N \beta \left[\sigma_{i_{s}}^l+ \sum\limits_{1\leq s' < s} \sigma_{i_{s'}}^l-\sigma_{i_{s'}}^l\circ T_{i_{s'}} \right] \right) \leq e^{2\beta |J|}.
\eea \] 
Using, this and the fact that
$$ \left|\Av \left(f_1-f_2 \right) \mathcal{E}\right| \leq \Av \left|f_1-f_2 \right| \mathcal{E}\leq \Av f'  \mathcal{E} $$
since $\left|f_1-f_2\right|\leq f' $ by assumption yields 
$$ \left|f_s-f_s\circ T_{i_s} \right| \leq  \left(1-e^{-2\beta}\right)e^{2\beta |J|} \Av f'  \mathcal{E}.$$
Therefore we can use (\ref{whatCdoes}) with $C \defi  \left(1-e^{-2\beta}\right)e^{2\beta |J|}$ to estimate (\ref{estimateviadecomp}) further. The upshot is 
$$\text{l.h.s. of } (\ref{exchange}) \leq \E \frac{1}{2}\sum\limits_{s=1}^{|J|}C \frac{\left(k+|J|\right)B+B^*}{N-1} + \PP(\Omega^c)\leq \E |J| C \frac{\left(k+|J|\right)B+B^*}{N} + \PP(\tilde \Omega^c) .$$
Using  $(1-x)^{k} \geq 1-kx$, which holds for any $k\in \mathbbm{N}_0$ and $x\in[0,1]$,
$$\PP(\tilde \Omega^c) = 1- \left(1- \frac{\gamma}{N} \right)^{k-1}\leq \frac{\gamma k}{N}.$$
Next are the estimates for $\E |J|e^{2\beta|J|}$ and $\E |J|^2e^{2\beta|J|}$. Recall that $|J| = \sum\limits_{i=1}^{N-1} g_{i,N}$, and that the $g_{i,N}$ are i.i.d. Bernoulli($\gamma/N$). 
It thus holds 
\[\bea
E |J|e^{2\beta|J|} & = (N-1) \frac{\gamma}{N} e^{2\beta} \left(1+\frac{\gamma}{N}(e^{2\beta}-1)\right)^{N-2} \\
&  \leq \gamma \exp\left(\gamma(e^{2\beta}-1) + 2\beta \right) =: a_{\beta, \gamma},
\eea\]
where the inequality uses that $(1+\frac{x}{N})^N\leq e^x$. Similar considerations yield
$$ \E |J|^2e^{2\beta|J|} \leq \frac{\gamma^2}{2} \exp\left(\gamma(e^{2\beta}-1) + 4\beta \right)=: b_{\gamma,\beta}.$$
With these estimates we see that 
\beq\label{laststepindep}\text{l.h.s } (\ref{exchange})\leq \frac{1}{N} \left[k \left(B C(\be, \gamma) + \gamma \right)+b_{\gamma,\beta}\left(1-e^{-2\beta}\right)B+C(\be, \gamma) B^*\right]\eeq
since $a_{\beta, \gamma} \left(1-e^{-2\beta}\right) \leq C(\be, \gamma)$. 
The proof of Lemma \ref{swap_one} is therefore concluded by setting
\[ 
B \defi \frac{\gamma}{1-C(\be, \gamma)}, \quad B^* \defi \frac{b_{\gamma,\beta}\left(1-e^{-2\beta}\right)B}{1-C(\be, \gamma)}.
\] 
\end{proof}

\begin{proof}[Proof of Proposition \ref{spin_ind}]
Let $C(\be,\gamma)<1$ and $f:\Sigma_k\rightarrow \R,f':\Sigma_k\rightarrow \R^+$ be functions having the three properties stated in the assumption of the Proposition. Remembering that we want to decouple one spin at a time we set 
$$f_m= f\left(\sigma_{1}^1,..,\sigma_{m}^m,\sigma_{m+1}^1,..,\sigma_{k}^1 \right) \quad \text{ and } \quad f'_m = f'\left(\sigma_{1}^1,..,\sigma_{m}^m,\sigma_{m+1}^1,..,\sigma_{k}^1 \right).$$ 
With this notation the left hand side of (\ref{spin_ind_estim}) is equal to  
\[
\E \left| \frac{\left< f_1\right>^{\otimes n}}{\left< f'_1\right>^{\otimes n}}-\frac{\left< f_k\right>^{\otimes n}}{\left< f'_k\right>^{\otimes n}}\right|\,,
\]
and using the triangle inequality after telescopic decomposition yields 
\[
\mbox{l.h.s.} (\ref{spin_ind_estim}) \leq \sum\limits_{m=1}^{k-1} \E\left| \frac{\left<f_m\right>^{\otimes n}}{\left<f'_m\right>^{\otimes n}}-\frac{\left<f_{m+1}\right>^{\otimes n}}{\left<f'_{m+1}\right>^{\otimes n}} \right|.
\]
Clearly, 
\[
\left|\frac{\left<f_m\right>^{\otimes n}}{\left<f'_m\right>^{\otimes n}}-\frac{\left<f_{m+1}\right>^{\otimes n}}{\left<f'_{m+1}\right>^{\otimes n}}\right| = \left|\frac{\left<f_m-f_{m+1}\right>^{\otimes n}}{\left<f'_m\right>^{\otimes n}}+\frac{\left<f_{m+1}\right>^{\otimes n}\left<f'_{m}-f'_{m+1}\right>^{\otimes n}}{\left<f'_{m+1}\right>_{-}^{\otimes n}\left<f'_{m}\right>^{\otimes n}}\right| \,,
\]
which again by the triangle inequality is {\it at most}
\[
\left|\frac{\left<f_m-f_{m+1}\right>^{\otimes n}}{\left<f'_m\right>^{\otimes n}}\right|+\left|\frac{\left<f_{m+1}\right>^{\otimes n}\left<f'_{m}-f'_{m+1}\right>^{\otimes n}}{\left<f'_{m+1}\right>_{-}^{\otimes n}\left<f'_{m}\right>^{\otimes n}}\right|\leq \left|\frac{\left<f_m-f_{m+1}\right>^{\otimes n}}{\left<f'_m\right>^{\otimes n}}\right|+\left|\frac{\left<f'_{m}-f'_{m+1}\right>^{\otimes n}}{\left<f'_{m}\right>^{\otimes n}}\right|,
\]
where the last estimate is due to $\left|f_{m+1}\right|\leq f'_{m+1}$. Expanding the terms by $C_1$,$C_2$ respectively, taking expectations and applying Lemma \ref{swap_one} settles the proof of Proposition \ref{spin_ind}.
\end{proof}

\subsection{Annealed decoupling}\label{proof_mag_ind}
In this section we prove that finitely many {\it magnetizations} are independent, $\nu_\star$-distributed random variables. Precisely: 

\begin{lem}\label{mag_ind}
Assume that $\be,\gamma_0\geq 0$.  For any $k \in \N$, it then holds:  
\beq\label{mag_ind_claim}
\sup\limits_{\gamma\leq \gamma_0} d\left(\mathcal{L} \left(\left<\s_i\right>\right)_{i\leq k},\nu_\star\left(\be,h,\gamma \right)^{\otimes k}\right)\leq D(\be,\gamma_0) \frac{k^3}{N},
\eeq
with the function $(\beta, \gamma) \mapsto D(\be,\gamma)$ increasing in both variables and finite for $C(\be, \gamma_0)<1$. 
\end{lem}
\begin{proof}
Let $\beta,\gamma_0,h\geq 0$ satisfy $C(\gamma_0,\be)<1$ and let $0\leq \gamma\leq\gamma_0$. The proof will be done by induction on $N$. Getting some technicalities out of the way first we note that the model is not well defined for $\gamma>N$ as then $p= \frac{\gamma}{N}>1$. Replacing in these cases $p$ by one will not harm any of the estimates we do for the induction step as they are all increasing in $\gamma$. With this convention, for $N=1$ the claim is trivial by picking $D(\be,\gamma)>1$ since $k$ can only be zero or one. For general $N$, $D(\be,\gamma)>2$ and $k\geq N/2$ the claim is also trivial as the left hand side of (\ref{mag_ind_claim}) is bounded by $k$ and the right hand side in this case is at least $k^2$. As for the interesting case, let $N\in \N$ and $k< N/2$. We set 
$$\mathcal{C} = \exp\left(\sum\limits_{j=1}^k\sigma_{N-k+j} \left(h-\beta\sum\limits_{i\leq N-k+j}g_{i,N-k+j} \sigma_i\right)\right).$$
We denote by $\left<\right>_{-k}$ the expectation w.r.t. $\mathcal{G}_{h,\beta,\frac{N-k}{N}\gamma,N-k}$, and  by $Y = \left(\left<\sigma_i\right>_{-k}\right)_{i\leq N-k}$  the vector of magnetizations on the $N-k$ system. (Considering the last $k$ spins instead of the first ones leads to lighter notation). What is absolutely crucial for the whole analysis is that $\mathcal C$ is independent of the randomness in $\left<\right>_{-k}$. We also notice that following (\ref{mean_one}) and (\ref{mean_two}) for $k$-many spins (instead of one) gives the identity 
\beq\label{selfsim}
\left<\sigma_j\right>= \frac{\left<\Av \sigma_j\mathcal{C}\right>_{-k}}{\left<\Av \mathcal{C}\right>_{-k}}
\eeq
for $N-k<j\leq N$. Here the average is taken over $\sigma_{N-k+1},..,\sigma_{N}\in \{0,1\}$.

For the remainder of this proof we set $X'$ to be a sequence of independent $\nu_\star\left(\be,h,\frac{N-k}{N}\gamma \right)$-distributed random variables. 

With the above notations, and by the triangle inequality, we have
\beq
 d\left(\mathcal{L} \left(\left<\s_i\right>\right)_{N-k< i \leq N},\nu_\star\left(\be,h,\gamma \right)^{\otimes k}\right)\leq \RM{1}+\RM{2}+\RM{3}\eeq
where 
\[ \bea
\RM{1} & \defi d\left(\mathcal{L} \left(\left<\s_j\right>\right)_{N-k< j \leq N},\mathcal{L} \left(\frac{\left<\Av \sigma_j\mathcal{C}\right>_Y}{\left<\Av \mathcal{C}\right>_Y}\right)_{N-k< j \leq N}\right), \\
\RM{2} & \defi d\left(\mathcal{L} \left(\frac{\left<\Av \sigma_j\mathcal{C}\right>_Y}{\left<\Av \mathcal{C}\right>_Y}\right)_{N-k< j \leq N},\mathcal{L} \left(\frac{\left<\Av \sigma_j\mathcal{C}\right>_{X'}}{\left<\Av \mathcal{C}\right>_{X'}}\right)_{N-k< j \leq N}\right), \\
\RM{3} &\defi d\left(\mathcal{L} \left(\frac{\left<\Av \sigma_j\mathcal{C}\right>_{X'}}{\left<\Av \mathcal{C}\right>_{X'}}\right)_{N-k< j \leq N},\nu_\star\left(\be,h,\gamma \right)^{\otimes k}\right)\,,
\eea \]
where all {\it all} averages are taken over $\sigma_{N-k+1},..,\sigma_{N}\in \{0,1\}$. The proof of  Lemma \ref{mag_ind} boils down to showing that: 
\begin{itemize}
\item[] $\RM{1}$ is 'small' by Proposition \ref{spin_ind}.
\item[] $\RM{2}$ is 'small' by the induction.
\item[] $\RM{3}$ is 'small' by construction of the $T$-operator.
\end{itemize}
(For the meaning of 'small', see below: \eqref{smallness_I}, \eqref{resII}  and \eqref{smallness_III} respectively).\\

A first estimate on the Monge-Kantorovich distance behind $\RM{1}$ is established by considering the coupling which is already given, and using \eqref{selfsim}. It holds:
\beq \label{RM1_eq1}
\RM{1} \leq \sum\limits_{j=1}^k\E\left|\frac{\left<\Av \sigma_{N-k+j}\mathcal{C}\right>_{-k}}{\left<\Av \mathcal{C}\right>_{-k}}- \frac{\left<\Av \sigma_{N-k+j}\mathcal{C}\right>_Y}{\left<\Av \mathcal{C}\right>_Y}\right|\,.
\eeq
To estimate the above we will use Lemma \ref{spin_ind}, which in turns requires a control of $\mathcal C$. To this end, let 
\beq \label{omega_1}
\Omega_1 \defi \left\{g_{i,j} = 0, \, \forall \, i,j>N-k\right\}
\eeq
be the event that there are no "direct" interactions between the last $k$ spins. Furthermore, let
\[
\Omega_2 \defi \left\{g_{i,j}=0 \text{ or } g_{i,j'}= 0 \; \forall i\leq N-k , j<j'>N-k\right\}
\] 
be the event that there are no interactions of the last $k$ spins "via" a single other spin. Finally, let
\[
\tilde \Omega := \Omega_1 \cap \Omega_2\,.
\]
(Remark that for fixed $k$ both $\tilde \Omega_1$ and $\tilde \Omega_2$ are likely to happen, for larger and larger $N$, and so is $\tilde \Omega$). 

We observe that all fractions appearing in (\ref{RM1_eq1}) are on $[0,1]$, which implies that 
the sum is bounded by $k$:  using this rough estimate, we thus obtain
\beq\label{RM1_eq2}
\RM{1} \leq \sum\limits_{j=1}^k\E \boldsymbol{1}_{\tilde \Omega}\left|\frac{\left<\Av \sigma_{N-k+j}\mathcal{C}\right>_{-k}}{\left<\Av \mathcal{C}\right>_{-k}}- \frac{\left<\Av \sigma_{N-k+j}\mathcal{C}\right>_Y}{\left<\Av \mathcal{C}\right>_Y}\right| + k\PP\left(\tilde \Omega^c\right)\,.
\eeq
Under the light of Lemma \ref{spin_ind}, we set $f \defi \boldsymbol{1}_{\tilde \Omega}\Av \sigma_{N-k+j}\mathcal{C}$ and $f' \defi \Av \mathcal{C}$. Using that on $\tilde \Omega$ each spin in $\mathcal C$ appears at most once, changing the value of one spin changes the exponent by $\beta$, at most. 
We apply Lemma \ref{spin_ind} with $C_1=C_2=e^\beta-1$. This yields 
\[
\RM{1} \leq 2 k \E S \left(e^\beta-1\right) \frac{S B+B^*}{N}   + k\PP\left(\tilde \Omega^c\right)
\]
where $S$ is the number of spins $\mathcal C$ depends on, to wit: 
\[
S= \sum\limits_{j= N-k+1}^N \sum\limits_{i=1}^{j-1} g_{i,j},
\] 
which is Binomial($kN+\frac{k(k+1)}{2},\frac{\gamma}{N} $)-distributed. Rough estimates on the first two moments of $S$, recalling that $k\leq \frac{N}{2}$, yield
\[
\E[S] \leq 2k\gamma \quad \mbox{ and } \quad \E[S^2]\leq 6k^2(\gamma+\gamma^2)
\]
It holds: 
\beq\label{Omega_calc}\PP\left(\tilde\Omega_1^c\right) = 1-\left(1-\frac{\gamma}{N}\right)^{k\left(k-1\right)}\leq \frac{\gamma k^2}{N} \eeq
\[ \PP\left(\tilde \Omega_2^c\right) \leq \sum\limits_{i=1}^{N-k}\sum\limits_{j=N-k+1}^{N}\sum\limits_{j'= j+1}^N\PP(g_{i,j}=g_{i,j'} = 1)\leq \frac{\gamma^2 k^2}{N} \]
and therefore $\PP(\tilde \Omega^c)\leq \frac{(\gamma+\gamma^2) k^2}{N}$. All in all, 
\beq \label{smallness_I}
\RM{1} \leq \left(e^\beta-1\right) \frac{ 12 (\gamma+\gamma^2) B+4 \gamma B^*}{N}k^3   + \frac{\left(\gamma+\gamma^2\right) k^3}{N} = :w_{\beta,\gamma}\frac{k^3}{N}\,,
\eeq
where $w_{\beta,\gamma}$ stands (here and throughout) for a constant depending on $\beta,\gamma$ only, which is increasing in both variables. \\

We next address $\RM{2}$ We estimate the Monge-Kantorovich distance using that the randomness in $\mathcal{C}$ is independent of $X',Y$ and taking the infimum over any coupling of $X'$ and $Y$: 
\[
\RM{2} \leq \inf \sum\limits_{j=1}^k  \E \left|\frac{\left<\Av \sigma_i\mathcal{C}\right>_Y}{\left<\Av \mathcal{C}\right>_Y} - \frac{\left<\Av \sigma_i\mathcal{C}\right>_X'}{\left<\Av \mathcal{C}\right>_X'} \right|
\] 
By the same estimate which leads to (\ref{RM1_eq2}), 
\beq\label{RM2_eq2}
\RM{2} \leq \inf \sum\limits_{j=1}^k  \E \left|\boldsymbol{1}_{\tilde\Omega}\frac{\left<\Av \sigma_i\mathcal{C}\right>_Y}{\left<\Av \mathcal{C}\right>_Y} - \boldsymbol{1}_{\tilde \Omega}\frac{\left<\Av \sigma_i\mathcal{C}\right>_X'}{\left<\Av \mathcal{C}\right>_X'} \right|+k\PP(\tilde \Omega^c).
\eeq
Consider the random functions 
\[
s_j:[0,1]^{N-k} \rightarrow [0,1], x\rightarrow \boldsymbol{1}_{\tilde \Omega}\frac{\left<\Av \sigma_{N-k+j}\mathcal{C}\right>_x}{\left<\Av \mathcal{C}\right>_x}.
\]
for $j\leq k$. On $\tilde \Omega$ the representation $\mathcal{C}= \prod\limits_{l=1}^k \mathcal{C}_l$ holds, where 
\[ 
\mathcal{C}_l \defi \exp\left(\sigma_{N-k+l} \left(h-\beta\sum\limits_{i\leq N-k}g_{i,N-k+l} \sigma_i\right)\right).
\]
We use this expand the $s$-functions: precisely we write:
\beq\label{DecoupOnOmega}
s_j(x) = \boldsymbol{1}_{\tilde \Omega}\frac{\left<\Av \sigma_{N-k+j} \prod\limits_{l=1}^k \mathcal{C}_l\right>_x}{\left<\Av \prod\limits_{l=1}^k \mathcal{C}_l\right>_x} = \boldsymbol{1}_{\tilde \Omega}\frac{\left<\Av_j \sigma_{N-k+j} \mathcal{C}_j \prod\limits_{l=1,l\neq j}^k \Av_l \mathcal{C}_l\right>_x}{\left<\prod\limits_{l=1}^k \Av_l \mathcal{C}_l\right>_x},
\eeq
where $\Av_l$ is the Average over $\sigma_{N-k+l}\in\{0,1\}$. Since $\left<.\right>_x$ is a product measure, and since the $C_l$ depend, on $\tilde \Omega$, on disjoint sets of $\sigma$, cancellations lead to 
\beq\label{DecoupOnOmega2}
s_j(x) = \boldsymbol{1}_{\tilde \Omega}\frac{\left<\Av_j \sigma_{N-k+j} \mathcal{C}_j \right>_x}{\left<\Av_j \mathcal{C}_j\right>_x}.
\eeq
Since $\mathcal C_j$ depends only on those $\sigma_i$ for which $g_{i,N-k+1} =1$, $s_j\left(x\right)$ only depends on those $x_i$. Consider the derivative in such a direction:
\beq\label{estderiv}\left| \partial_{x_i} s_j(x)  \right| = \boldsymbol{1}_{\tilde \Omega}\frac{\left|\frac{\partial}{\partial x_i}\left(e^h \prod\limits_{l\leq N-k}\left< \exp\left(-\beta g_{l,j}\sigma_l\right)\right>_{x}\right)\right|}{\left(1+e^h \prod\limits_{l\leq N-k}\left< \exp\left(-\beta g_{l,j}\sigma_l\right)\right>_{x}\right)^{2}}.
\eeq
The numerator is given by
$$ \left|e^h \left(e^{-\beta}-1\right)\prod\limits_{l\leq N-k, l\neq i}\left< \exp\left(-\beta g_{l,j}\sigma_l\right)\right>_{x} \right|\leq \left(e^{\beta}-1\right)e^h \prod\limits_{l\leq N-k}\left< \exp\left(-\beta g_{l,j}\sigma_l\right)\right>_{x},
$$
hence the following estimate holds 
\beq\label{derivestim} \| \partial_{x_j} s_j \|_{\infty} \leq \boldsymbol{1}_{\tilde \Omega}\textbf{1}_{\{g_{i,j} = 1\}}\left(e^{\beta}-1\right) \sup\limits_{t\geq 0} \frac{t}{(1+t)^2} \leq\boldsymbol{1}_{\tilde \Omega}\textbf{1}_{\{g_{i,j} = 1\}} \frac{e^{\beta}-1}{4} .\eeq 
Using this in \eqref{RM2_eq2},
\[ \bea
\RM{2} - k\PP(\tilde \Omega^c) & \leq \inf \sum\limits_{j=N-k+1}^N  \E\boldsymbol{1}_{\tilde \Omega}\sum\limits_{i=1}^{N-k} \left|\frac{\partial s_j}{\partial x_i}\right|_{\infty}\left|X_i-Y_i\right| \\
& \leq \frac{e^{\beta}-1}{4} \inf \sum\limits_{i=1}^{N-k} \E \boldsymbol{1}_{\tilde \Omega}\sum\limits_{j=1}^k\textbf{1}_{\{g_{i,N-k+j} = 1\}} \left|X_i-Y_i\right| \,.
\eea \]
Introduce now the event $\mathcal{A}_i \defi \{\exists j\leq k: g_{i,N-k+j} = 1\}$. On $\tilde \Omega$
it plainly holds that 
\[
\sum\limits_{j=1}^k\textbf{1}_{\{g_{i,N-k+j} = 1\}} = \textbf{1}_{{\mathcal A}_i}\,,
\]
and since the newly introduced $\mathcal{A}$-events are independent, and independent of $X'$ and $Y$, we get
\[
\RM{2} - k\PP(\tilde \Omega^c)\leq  \frac{e^{\beta}-1}{4} \inf \sum\limits_{i=1}^{N-k} \E \textbf{1}_{\mathcal{A}_i} \left|X_i-Y_i\right|\,.
\]
By the induction assumption, conditionally on all $\mathcal{A}_i$ for $i \leq N-k$, we have
\[
\RM{2} \leq  \frac{e^{\beta}-1}{4}  D(\be,\gamma) \frac{\E \left(\sum\limits_{i=1}^{N-k}\textbf{1}_{\mathcal{A}_i}\right)^3}{N-k}+ k\PP(\tilde \Omega^c),
\]
since $\frac{N-k}{N}\gamma < \gamma_0$. 
We now observe that 
\[
\sum\limits_{i=1}^{N-k}\textbf{1}_{\mathcal{A}_i} \stackrel{(d)}{=}  \text{Binomial}\left(N-k, 1-\left(1-\frac{\gamma}{N}\right)^k \right), 
\]
hence, by simple estimates,
\[
\E \left(\sum\limits_{i=1}^{N-k}\textbf{1}_{\mathcal{A}_i}\right)^3 \leq \left(\gamma^3+3\gamma^2+\gamma\right)k^3
\]
Recalling the estimates on $\PP\left(\tilde \Omega^c\right)$, since $k\leq N/2$, we therefore have
\beq \bea \label{resII}
\RM{2} & \leq \frac{e^{\beta}-1}{2}\left(\gamma^3+3\gamma^2+\gamma\right)  D(\be,\gamma) \frac{k^3}{N}+  \frac{\left(\gamma+\gamma^2\right) k^3}{N} \\
& \leq C(\be,\gamma)D(\be,\gamma) \frac{k^3}{N} +  \frac{\left(\gamma+\gamma^2\right) k^3}{N}\,.
\eea \eeq
We next move to $\RM{3}$ By the triangle inequality 
\[ \bea
\RM{3} & \leq d\left(\mathcal{L} \left(\frac{\left<\Av \sigma_j\mathcal{C}\right>_{X'}}{\left<\Av \mathcal{C}\right>_{X'}}\right)_{N-k< j \leq N},\nu_\star\left(\be,h, \frac{N-k}{N}\gamma  \right)^{\otimes k}\right)\\
& \hspace{4cm} + d\left(\nu_\star\left(\be,h, \frac{N-k}{N}\gamma  \right)^{\otimes k},\nu_\star\left(\be,h,\gamma  \right)^{\otimes k}\right).
\eea\]
Consider now independent random variables $Z_1,..,Z_k$ which are $\nu_\star\left(\be,h,\gamma  \right)^{\otimes k}$-distributed. (Remark that the $Z$'s may depend on the randomness appearing in $X',\mathcal{C}$: a concrete choice will be given only later, see \eqref{uno} and \eqref{due} below). By definition of the Monge-Kantorovich distance we have that
\[ \bea
& d\left(\mathcal{L} \left(\frac{\left<\Av \sigma_i\mathcal{C}\right>_{X'}}{\left<\Av \mathcal{C}\right>_{X'}}\right)_{N-k< i \leq N},\nu_\star\left(\be,h, \frac{N-k}{N}\gamma  \right)^{\otimes k}\right) \\
& \hspace{4cm} \leq \sum\limits_{j=1}^k \E \left|\frac{\left<\Av \sigma_{N-k+j}\mathcal{C}\right>_{X'}}{\left<\Av \mathcal{C}\right>_{X'}} - Z_j\right| \\
& \hspace{4cm} \leq \sum\limits_{j=1}^k \E \boldsymbol{1}_{\tilde \Omega}\left|\frac{\left<\Av \sigma_{N-k+j}\mathcal{C}\right>_{X'}}{\left<\Av \mathcal{C}\right>_{X'}} - Z_j\right| + k\PP(\tilde \Omega^c), 
\eea \]
the last step by restricting to $\tilde \Omega$. Recall from \eqref{DecoupOnOmega} that 
\beq\label{RM3_calc1}
\boldsymbol{1}_{\tilde \Omega}\frac{\left<\Av \sigma_{N-k+j}\mathcal{C}\right>_{X'}}{\left<\Av \mathcal{C}\right>_{X'}} =  \boldsymbol{1}_{\tilde \Omega}\frac{\left<\Av_j \sigma_{N-k+j} \mathcal{C}_j \right>_{X'}}{\left<\Av_j \mathcal{C}_j\right>_{X'}}
\eeq
holds and that on $\tilde \Omega$ the $\mathcal{C}_j$ depend on different $\sigma_i$. Therefore the right hand side depends for each $j$ on different $X_i$. Computing the averages, and plugging in the definition of $\mathcal{C}_j$, leads to 
\[
(\ref{RM3_calc1}) = \boldsymbol{1}_{\tilde \Omega}\frac{\left<\exp\left(h-\beta\sum\limits_{i\in J_j} \sigma_i\right)\right>_{X'}}{1+\left<\exp\left(h-\beta\sum\limits_{i\in J_j} \sigma_i\right)\right>_{X'}} = \boldsymbol{1}_{\tilde \Omega}\left(1+ \left<\exp\left(h-\beta\sum\limits_{i\in J_j} \sigma_i\right)\right>_{X'}^{-1}\right)^{-1}
\]
where $J_j = \{i\leq N-k: g_{i,N-k+j} = 1\}$ are disjoint sets on $\tilde \Omega$. We now consider $r_j$ to be a Poisson($\frac{N-k}{N}\gamma$)-distributed random variables independent of each other and $X'$, but optimally coupled to $|J_j|$. This is possible since the $|J_j|$ are independent of each other. If $\tilde \Omega$ occurs and $r_j = |J_j|$, which are events independent of $X'$ we set 
\beq \label{uno}
Z_j = \left(1+ \left<\exp\left(h-\beta\sum\limits_{i\in J_j} \sigma_i\right)\right>_{X'}^{-1}\right)^{-1},
\eeq
otherwise we set
\beq \label{due}
Z_j = \left(1+ \left<\exp\left(h-\beta\sum\limits_{i\leq r_j} \sigma_i\right)\right>_{X_j}^{-1}\right)^{-1}\,,
\eeq
where $X_j= \left(X_{j,1},X_{j,2},...\right)$ is a sequence of independent $\nu_\star\left(\be,h,\gamma  \right)$ distributed random variables independently of $X'$ and of $X_l$ for $l\neq j$. With this, the $Z_j$ are independent 
random variables, with identical distribution given by
$$
T_{\be,h,\frac{N-k}{N}\gamma}\nu_\star\left(\be,h,\frac{N-k}{N}\gamma  \right) = \nu_\star\left(\be,h,\frac{N-k}{N}\gamma  \right).
$$
It then holds: 
\[
\boldsymbol{1}_{\tilde \Omega}\left|\frac{\left<\Av \sigma_{N-k+j}\mathcal{C}\right>_{X'}}{\left<\Av \mathcal{C}\right>_{X'}} - Z_j\right| \leq \boldsymbol{1}_{\{|J_j|\neq r_j\}}\,,
\]
since the term is zero on $\{|J_j| = r_j\}$ and bounded by one no-matter-what. Collecting all estimates we thus have 
\beq\label{RM3_calc2}
\RM{3} \leq  \sum\limits_{j=1}^k \PP\left(|J_j|\neq r_j\right) +k\PP(\tilde{\Omega}^c)+d\left(\nu_\star\left(\be,h, \frac{N-k}{N}\gamma  \right)^{\otimes k},\nu_\star\left(\be,h,\gamma  \right)^{\otimes k}\right)
\eeq
Taking the infimum only coordinate by coordinate, the rightmost term above is at most
\[ 
k d\left(\nu_\star\left(\be,h, \frac{N-k}{N}\gamma  \right),\nu_\star\left(\be,h,\gamma  \right)\right) 
\leq k \left|\frac{N-k}{N}\gamma-\gamma \right| = \gamma \frac{k^2}{N} \leq \gamma \frac{k^3}{N}, 
\]
the first inequality by Proposition \ref{prop_contraction}. 

As for the first term on the r.h.s. of \eqref{RM3_calc2}, by the optimality of the coupling and since all summands are identical, we see that it equals, in fact, $\frac{k}{2} d_{TV}(\mathcal{L}(|J_1|),\mathcal{L}(r_1))$. Since $|J_j|$ is Binomial$\left(N-k,\frac{\gamma}{N}\right)$-distributed and $r_j$ is Poisson($\frac{N-k}{N}\gamma$)-distributed, their total variation is, by well-known estimates, at most 
\[ 
\frac{\gamma^2 \left(N-k\right)}{N^2} \leq \gamma^2 \frac{k^3}{N}.
\]
The middle term in \eqref{RM3_calc2} is bounded by $\left(\gamma+\gamma^2\right)\frac{k^3}{N}$ by \eqref{Omega_calc} and ff.. All in all, we have
\beq \label{smallness_III}
\RM{3} \leq 2\left(\gamma+\gamma^2 \right) \frac{k^3}{N}\,.
\eeq
Putting together the estimates \eqref{smallness_I}, \eqref{resII} and \eqref{smallness_III}, we thus have that
\[
\RM{1}+\RM{2}+\RM{3} \leq  \left(C(\be,\gamma)D(\be,\gamma)  + 3\gamma+3\gamma^2+ w_{\beta,\gamma}\right)\frac{k^3}{N}\,.
\]
The above holds for {\it any} choice of $D$ "inherited" from the induction step, but we now specify a concrete choice: we let 
\[
D(\be,\gamma) \defi \max\left\{2, \frac{3\gamma+3\gamma^2 + w_{\beta,\gamma}}{1-C(\be,\gamma)} \right\}.
\]
(It is immediate to check that this function satisfies the required monotonicity). 

The proof of  Lemma \ref{mag_ind} is therefore concluded by observing that 
\[
C(\be,\gamma)D(\be,\gamma)  + 3\gamma+3\gamma^2+ w_{\beta,\gamma} \leq D(\be,\gamma_0).
\]
\end{proof}
\vspace{0.6cm}

\subsection{Proof of Theorem \ref{full_ind}}  \label{finish_full_ind}
Let $h,\beta,\gamma \geq 0$ with $C(\be, \gamma) < 1$ and  $k\leq N\in \N$. Consider a function $f: \{0,1\}^k \rightarrow \R$ and a Lipschitz continuous function $g:\left[\min  f, \max f\right] \rightarrow \R$. By the triangle inequality, 
\[ 
\left|\E g\left(\left<f\left(\sigma_1,..,\sigma_k\right)\right>\right) - \E g\left(\E\left[f\left(B_1,..,B_k\right) | X\right]\right) \right|\leq \RM{1}+\RM{2},
\]
where 
\[
\RM{1} \defi \E \left| g\left(\left<f\left(\sigma_1,..,\sigma_k\right)\right>\right) - g\left(\left<f\left(\sigma_1,..,\sigma_k\right)\right>_Y\right) \right|,
\]
\[
\RM{2} \defi \left| \E g\left(\left<f\left(\sigma_1,..,\sigma_k\right)\right>_Y\right) - \E g\left(\E\left[f\left(B_1,..,B_k\right) | X\right]\right) \right|,
\]
and $Y \defi \left(\left<\s_i\right>\right)_{i\leq N}$.

As for $\RM{1}$, since $g$ is $L_g$-Lipschitz, 
\[
\RM{1}\leq L_g \E \left| \left<f\left(\sigma_1,..,\sigma_k\right)\right> - \left<f\left(\sigma_1,..,\sigma_k\right)\right>_Y \right|.
\]
Therefore Lemma \ref{spin_ind} implies, with $f' \defi \max|f|$, $C_1\defi 2$ and $C_2\defi 0$, that 
\beq \label{estimateI}
\RM{1}\leq L_g  2k  \|f \|_\infty \frac{kB+B^*}{N}\leq L_g \|f \|_\infty \left(2B+2B^*\right)\frac{k^3}{N},
\eeq
where the second estimate simply uses that $k \in \N$. \\

As for $\RM{2}$, we compute the conditional expectation
\[ \bea
\E\left[f\left(B_1,..,B_k\right) | X\right] & = \sum\limits_{\sigma\in\Sigma_k} \left(\prod\limits_{i=1}^k \PP(B_i = \sigma_i|X)\right)f\left(\sigma_1,..,\sigma_k\right) \\
&  = \sum\limits_{\sigma\in\Sigma_k}\prod\limits_{i=1}^k  \left(X_i \boldsymbol{1}_{\sigma_i=1}+\left(1-X_i\right) \boldsymbol{1}_{\sigma_i=0} \right)f\left(\sigma_1,..,\sigma_k\right) = \left< f\left(\sigma_1,..,\sigma_k\right)\right>_X.
\eea \]
Hence, by the Lipschitz-continuity of $g$, and for any coupling of $X$ and $Y$, it holds: 
\[
\RM{2} \leq L_g \E \left| \left<f\left(\sigma_1,..,\sigma_k\right)\right>_Y - \left< f\left(\sigma_1,..,\sigma_k\right)\right>_X \right|
\]
Consider the function $s: [0,1]^k \rightarrow \R, x \rightarrow \left<f\left(\sigma_1,..,\sigma_k\right)\right>_x$. One easily sees that  $\| \partial_{x_i} s\|_\infty \leq 2 \| f \|_\infty$. Using this, 
\[
\RM{2} \leq L_g \E  \sum_{i=1}^{k} \| \partial_{x_i} s\|_\infty \left|Y_i-X_i \right|\leq 2 L_g  \|f\|_\infty \E  \sum_{i=1}^{k}\left|Y_i-X_i \right|.
\]
Since we considered an arbitrary coupling of $X$ and $Y$ the inequality holds still true as we take the infimum over all couplings. This yields
\[
\RM{2} \leq 2 L_g  \|f\|_\infty  \inf \E  \sum_{i=1}^{k}\left|Y_i-X_i \right| = 2 L_g \|f \|_\infty d\left(\mathcal{L}Y, \mathcal{L}X\right),
\]
by the definition of the Monge-Kantorovich distance. Plugging in the distributions of $X$ and $Y$ gives
\[
\RM{2} \leq 2 L_g \|f \|_\infty d\left(\mathcal{L}\left(\left<\sigma_i\right>\right)_{i\leq k}, \nu_\star\left(\be,h,\gamma \right)^{\otimes k}\right).,
\]
By Lemma \ref{mag_ind}, and \eqref{estimateI}, we obtain 
\beq
\RM{1}+\RM{2} \leq L_g \|f\|_\infty \left(2B+2B^*\right)\frac{k^3}{N} + 2 L_g \|f \|_\infty D(\be,\gamma) \frac{k^3}{N}.
\eeq
This, together with the $(\beta, \gamma)$-monotonicity of $B$,$B^*$ and $D(\be,\gamma)$, settles 
the proof of Theorem \ref{full_ind}. \\
${}$ \hfill $\square$

\section{Appendix}
We give here a proof of Fact \ref{interpolation}, together with some technical estimates on Binomial-distributions. 
\begin{lem}\label{lem_interp1}
The time-derivative of the interpolating free energy is given by
\[
\varphi'(t) = \frac{\gamma}{2}\left(\frac{1}{N^2}\sum\limits_{i,j\leq N}\E \log \left<\exp\left( -\beta \sigma_i \sigma_j\right)\right>_t -\frac{2}{N}\sum\limits_{i=1}^N  \E \log \left<\exp\left(\sigma_i \log \left<e^{-\beta\varepsilon}\right>_{X^{\boldsymbol{\delta}}_{i,1}}\right)\right>_t\right)+o(1)
\]
\end{lem}

\begin{proof}
We lighten notation by setting $\chi^{\boldsymbol{\delta}}_{i,j} = \log \left<e^{-\beta\varepsilon}\right>_{X^{\boldsymbol{\delta}}_{i,j}}$. 

It holds:
\begin{equation}\label{derivcalc}\frac{1}{u} \left(\varphi(t+u)-\varphi(t)\right) = \frac{1}{N u}\E \log \frac{\sum\limits_{{\boldsymbol{\delta}}\in\N^K}\sum\limits_{\sigma \in \Sigma_N} v_{\boldsymbol{\delta}} \exp\left(H^{\boldsymbol{\delta}}_{N,t+u}(\sigma)\right)}{\sum\limits_{{\boldsymbol{\delta}}\in\N^K}\sum\limits_{\sigma \in \Sigma_N} v_{\boldsymbol{\delta}} \exp\left(H^{\boldsymbol{\delta}}_{N,t}(\sigma)\right)} \end{equation}
where the joint distribution of $H^{\boldsymbol{\delta}}_{N,t+u}(\sigma)$ and $H^{\boldsymbol{\delta}}_{N,t}(\sigma)$ can be chosen in any way that does not touch the marginals.
We write $g^{*t+u}$ for the $g^*$ in $H^{\boldsymbol{\delta}}_{N,t+u}(\sigma)$ and $g^{*t}$ for the $g^*$ in $H^{\boldsymbol{\delta}}_{N,t}(\sigma)$ to distinguish them and analogously for $\hat{g}$.
We set 
$$ g^{*t+u}_{i,j}= g^{*t}_{i,j}+ \left(1-g^{*t}_{i,j}\right)b^*_{i,j} \mbox{ and } \hat{g}^{t}_{i,j}= \hat{g}^{t+u}_{i,j}+\left(1-\hat{g}^{t+u}_{i,j}\right)\hat{b}_{i,j}$$
where the $b^*$ are independent Bernoulli $(\frac{\gamma u}{N-\gamma t})$ random variables and the $\hat{b}$ are independent Bernoulli $(u \frac{\gamma}{N-\gamma\left(1-t-u\right)})$ random variables. $b^*$,$\hat{b}$ are chosen independently and independent of any other randomness in $H^{\boldsymbol{\delta}}_{N,t}(\sigma)$,$H^{\boldsymbol{\delta}}_{N,t+u}(\sigma)$. One easily checks that $g^{*t+u}$ and $\hat{g}^{t}$ have the correct distribution. With this construction we have 
$$ H^{\boldsymbol{\delta}}_{N,t+u}(\sigma)+ \beta \sum\limits_{i<j} \left(1-g^{*t}_{i,j}\right)b^*_{i,j} \sigma_i \sigma_j  = H^{\boldsymbol{\delta}}_{N,t}(\sigma) - \sum\limits_{i,j=1}^N \left(1-\hat{g}^{t+u}_{i,j}\right)\hat{b}_{i,j}\sigma_i \chi^{\boldsymbol{\delta}}_{i,j}=: \tilde{H}(\sigma)$$
Expanding the fraction in equation (\ref{derivcalc}) by the partition function of $\tilde{H}$ yields
\[ \bea
& \frac{1}{N u}\E \log \frac{\sum\limits_{{\boldsymbol{\delta}}\in\N^K}\sum\limits_{\sigma \in \Sigma_N} v_{\boldsymbol{\delta}} \exp\left(H^{\boldsymbol{\delta}}_{N,t+u}(\sigma)\right)}{\sum\limits_{{\boldsymbol{\delta}}\in\N^K}\sum\limits_{\sigma \in \Sigma_N} v_{\boldsymbol{\delta}} \exp\left(\tilde{H}(\sigma)\right)}-\frac{1}{N u}\E \log \frac{\sum\limits_{{\boldsymbol{\delta}}\in\N^K}\sum\limits_{\sigma \in \Sigma_N} v_{\boldsymbol{\delta}} \exp\left(H^{\boldsymbol{\delta}}_{N,t}(\sigma)\right)}{\sum\limits_{{\boldsymbol{\delta}}\in\N^K}\sum\limits_{\sigma \in \Sigma_N} v_{\boldsymbol{\delta}} \exp\left(\tilde{H}(\sigma)\right)} \\
& \qquad \qquad =\frac{1}{N u}\E \log\left<\exp\left(-\beta \sum\limits_{i<j} \left(1-g^{*t}_{i,j}\right)b^*_{i,j} \sigma_i \sigma_j\right)\right>_{\tilde{H}} \\
& \hspace{3cm} -\frac{1}{N u}\E \log \left<\exp\left( \sum\limits_{i,j=1}^N \left(1-\hat{g}^{t+u}_{i,j}\right)\hat{b}_{i,j}\sigma_i \chi^{\boldsymbol{\delta}}_{i,j}\right)\right>_{\tilde{H}}. 
\eea \]
The event that more then one of the $b$'s is $1$ has probability of order $u^2$, and can therefore be neglected
in the limit $u\to 0$. On the other hand, if all $b$ are zero, the expressions in the expectations also vanish. It follows that the above equals
$$ \frac{1}{N u}\sum\limits_{i<j} \PP\left(b^*=\hat{b}=0 \mbox{ except for } b^*_{i,j}=1\right)\E \log \left<\exp\left( -\beta \left(1-g^{*t}_{i,j}\right) \sigma_i \sigma_j\right)\right>_{\tilde{H}}$$
$$ - \frac{1}{N u}\sum\limits_{i,j=1}^N \PP\left(b^*=\hat{b}=0 \mbox{ except for } \hat{b}_{i,j}=1\right)\E \log \left<\exp\left(  \left(1-\hat{g}^{t+h}_{i,j}\right)\sigma_i \chi^{\boldsymbol{\delta}}_{i,j}\right)\right>_{\tilde{H}}+o_u(1). $$
Computing the probabilities we see that the first probability is equal to $\frac{u\gamma}{N-t\gamma}+o_u(1)$ and the second probability is equal to $\frac{u \gamma}{N-\gamma\left(1-t\right)}+o_u(1)$. Since there are $\frac{N(N-1)}{2}$ respectively $N^2$ summands,  taking the $u\rightarrow \infty$ limit we obtain 
$$ \varphi'(t) = \frac{N-1}{N-t\gamma} \frac{\gamma}{2}\E \log \left<\exp\left( -\beta \left(1-g^{*t}_{1,2}\right) \sigma_1 \sigma_2\right)\right>_t$$
$$ - \frac{N}{N-\gamma\left(1-t\right)}\gamma\E \log \left<\exp\left(\left(1-\hat{g}^{t}_{1,2}\right)\sigma_1 \chi^{\boldsymbol{\delta}}_{1,2}\right)\right>_t$$
Observe that replacing  $\left(1-g^{*t}_{1,2}\right)$ and the corresponding $\hat{g}$-term by one has a vanishing contribution in the large $N$-limit, hence 
\beq\label{appendix_calc_1}
\varphi'(t) = \frac{\gamma}{2}\left(\E \log \left<\exp\left( -\beta \sigma_1 \sigma_2\right)\right>_t -2\E \log \left<\exp\left( \sigma_1 \chi^{\boldsymbol{\delta}}_{1,2}\right)\right>_t\right)+o(1)\,.
\eeq
The Hamiltonian only depends on $\chi^{\boldsymbol{\delta}}_{1,2}$ when $\hat{g}_{1,2}=1$, which happens with probability of order $N^{-1}$. Therefore and by the boundedness of the second term in \eqref{appendix_calc_1}
\[
\left|\E \log \left<\exp\left( \sigma_1 \chi^{\boldsymbol{\delta}}_{1,2}\right)\right>_t- \E \log \left<\exp\left( \sigma_1 \chi^{\boldsymbol{\delta}}\right)\right>_t\right| = o(1).
\]
Consequently we have
$$  \varphi'(t) = \frac{\gamma}{2}\left(\E \log \left<\exp\left( -\beta \sigma_1 \sigma_2\right)\right>_t - 2\E \log \left<\exp\left( \sigma_1 \chi^{\boldsymbol{\delta}}\right)\right>_t\right)+o(1)$$
and by symmetry among sites
$$  \varphi'(t) = \frac{\gamma}{2}\left(\frac{1}{N^2}\sum\limits_{i,j\leq N}\E \log \left<\exp\left( -\beta \sigma_i \sigma_j\right)\right>_t -\frac{2}{N}\sum\limits_{i=1}^N  \E \log \left<\exp\left(\sigma_i \chi^{\boldsymbol{\delta}}\right)\right>_t\right)+o(1)$$
as the diagonal has only vanishing contribution. Plugging in $\chi^{\boldsymbol{\delta}}$ gives the result.
\end{proof}

\begin{proof}[Proof of Fact \ref{interpolation}]
Using Lemma \ref{lem_interp1} and adopting the notation therein introduced,  
\[
\varphi'(t) = \frac{\gamma}{2} \E\left(\RM{1} + \RM{2}\right)+o(1), 
\] 
where
\[
\RM{1} \defi \frac{1}{N^2}\sum\limits_{i,j\leq N} \log \left<\exp\left( -\beta \sigma_i \sigma_j\right)\right>_t, \quad 
\RM{2} \defi -\frac{2}{N}\sum\limits_{i=1}^N  \log \left<\exp\left(\sigma_i \chi^{\boldsymbol{\delta}}\right)\right>_t
\]
It holds:
\[ \bea
\RM{1}¨&= \frac{1}{N^2}\sum\limits_{i,j\leq N}\log \left[1-\left(1-e^{-\beta}\right)\left<\sigma_i\sigma_j\right>_t\right] \\
& = - \frac{1}{N^2}\sum\limits_{i,j\leq N}\sum\limits_{n=1}^{\infty} \frac{\left(e^{-\beta}-1\right)^n}{n}\left<\sigma_i\sigma_j\right>^n.
\eea \]
Using replicas, we reformulate the above  as 
\[\bea 
\RM{1} & = - \frac{1}{N^2}\sum\limits_{i,j\leq N}\sum\limits_{n=1}^{\infty} \frac{\left(e^{-\beta}-1\right)^n}{n}\left<\prod\limits_{l=1}^n \sigma^l_i\sigma^l_j\right>_t^{\otimes n} \\
& =-\sum\limits_{n=1}^{\infty} \frac{\left(e^{-\beta}-1\right)^n}{n} \left< \left(\frac{1}{N}\sum\limits_{i\leq N}\prod\limits_{l=1}^n \sigma^l_i\right)^2\right>_t^{\otimes n}. 
\eea \]
As for the second term, denoting the expectation with respect to all $X^{{\boldsymbol{\delta}}^l}, X_1^{{\boldsymbol{\delta}}}$ and $X_2^{\boldsymbol{\delta}}$ by $\E_X$, we have: 
\[\bea 
\E_X\RM{2} & = -\frac{2}{N}\sum\limits_{i\leq N} \E_X\log \left<\exp\left(\sigma_i \chi^{\boldsymbol{\delta}}\right)\right>_t  = -\frac{2}{N}\sum\limits_{i\leq N}\E_X\log \left<\left<\exp\left(-\beta\varepsilon\sigma_i \right)\right>_{X^{\boldsymbol{\delta}}}\right>_t \\
& = -\frac{2}{N}\sum\limits_{i\leq N}\E_X\log  \left[1-\left(1-e^{-\beta}\right)\left<X^{\boldsymbol{\delta}} \sigma_i\right>_t\right] =\frac{2}{N}\sum\limits_{i\leq N}\sum\limits_{n=1}^{\infty} \frac{\left(e^{-\beta}-1\right)^n}{n}\E_X\left< X^{\boldsymbol{\delta}} \sigma_i\right>_t^n \\
& =\sum\limits_{n=1}^{\infty} \frac{\left(e^{-\beta}-1\right)^n}{n}\left<2\E_X\prod\limits_{l=1}^n X^{{\boldsymbol{\delta}}^l} \frac{1}{N}\sum\limits_{i\leq N}\prod\limits_{l\leq n}\sigma^l_i\right>_t^{\otimes n}\,.
\eea \]
We set  
$$ \RM{3} = \log \sum\limits_{{\boldsymbol{\delta}}\in\N^K} v_{\boldsymbol{\delta}}  \left<e^{-\beta\varepsilon_1\varepsilon_2}\right>_{X^{\boldsymbol{\delta}}} \,.$$
Performing analogous computations to the ones for $\RM{1}$ and $\RM{2}$, we get
\[\bea 
\E_X\RM{3} & = \E_X\log \left< \left<e^{-\beta\varepsilon_1\varepsilon_2}\right>_{\left(X^{\boldsymbol{\delta}}_1,X^{\boldsymbol{\delta}}_2\right)} \right>_t \\
& =\E_X\log \left[1-\left(1-e^{-\beta}\right)\left<X^{\boldsymbol{\delta}}_1 X^{\boldsymbol{\delta}}_2\right>_t\right] = -\sum\limits_{n=1}^{\infty} \frac{\left(e^{-\beta}-1\right)^n}{n}\E_X\left< X_1^{\boldsymbol{\delta}} X_2^{\boldsymbol{\delta}} \right>_t^n \\
&= -\sum\limits_{n=1}^{\infty} \frac{\left(e^{-\beta}-1\right)^n}{n}\left<\E_X \prod\limits_{l\leq n}X^{{\boldsymbol{\delta}}^l}_1 X^{{\boldsymbol{\delta}}^l}_2\right>^{\otimes n}_t = -\sum\limits_{n=1}^{\infty} \frac{\left(e^{-\beta}-1\right)^n}{n}\left<\left(\E_X\prod\limits_{l\leq n}X^{{\boldsymbol{\delta}}^l}\right)^2\right>^{\otimes n}_t 
\eea \]
Collecting all terms we obtain
\beq\label{deriv_rep}\varphi'(t)+\frac{\gamma}{2} \E \RM{3} = -\frac{\gamma}{2} \sum\limits_{n=1}^{\infty} \frac{\left(e^{-\beta}-1\right)^n}{n} \E\left<\left(\frac{1}{N}\sum\limits_{i\leq N}\prod\limits_{l=1}^n \sigma^l_i- \E_X\prod\limits_{l\leq n}X^{{\boldsymbol{\delta}}^l}\right)^2 \right>^{\otimes n}_t + o(1).\eeq
By $\eqref{deriv_rep}$, as $\RM{3}$ does not depend on $t$ and since $f_N(\be, h, \gamma) = \varphi(1)$ we have 
\beq\label{interpolation_calc1}
f_N(\be, h, \gamma) = \varphi(0) + \int\limits_{0}^{1} \varphi'(t) = \varphi(0)- \frac{\gamma}{2}\E \RM{3} + R_{N, \be, h, \gamma}(\zeta, K,  \mathcal V_{\boldsymbol m})+ o(1).
\eeq
We rearrange
\[
\sum\limits_{\sigma \in \Sigma_N}\exp\left(H^{\boldsymbol{\delta}}_{N,0}(\sigma)\right) =\sum\limits_{\sigma \in \Sigma_N}\exp\left( \sum\limits_{i=1}^N \sigma_i \left(h+\sum\limits_{j=1}^N\hat{g}_{i,j}\log \left<e^{-\beta\varepsilon}\right>_{X^{\boldsymbol{\delta}}_{i,j}}\right)\right)
\]
\[
=\prod\limits_{i=1}^{N}\left(1+ \exp\left(h+\sum\limits_{j=1}^N\hat{g}_{i,j}\log \left<e^{-\beta\varepsilon}\right>_{X^{\boldsymbol{\delta}}_{i,j}}\right)\right) = \prod\limits_{i=1}^{N}\left(1+ e^h \prod\limits_{j=1:\hat{g}_{i,j} = 1}^N \left<e^{-\beta\varepsilon}\right>_{X^{\boldsymbol{\delta}}_{i,j}}\right)\,,
\]
and therefore 
\[ \bea
\varphi(0) &= \frac{1}{N}\E \log \sum\limits_{{\boldsymbol{\delta}}\in\N^K}\sum\limits_{\sigma \in \{0,1\}^N} v_{\boldsymbol{\delta}} \exp\left(H^{\boldsymbol{\delta}}_{N,0}(\sigma)\right)\\
& =\E \log \sum\limits_{{\boldsymbol{\delta}}\in\N^K}v_{\boldsymbol{\delta}} \left(1+ e^h \prod\limits_{j=1:\hat{g}_{i,j} = 1}^N \left<e^{-\beta\varepsilon}\right>_{X^{\boldsymbol{\delta}}_{1,j}}\right),
\eea \]
using the symmetry in distribution.  Now the expectation depends only on the (random) number of factors in the product, which converges weakly to the Poisson($\gamma$) distribution. Hence, by standard compactness arguments the above equals
\[
\E \log \sum\limits_{{\boldsymbol{\delta}}\in\N^K}v_{\boldsymbol{\delta}} \left(1+ e^h \prod\limits_{j=1}^r \left<e^{-\beta\varepsilon}\right>_{X^{\boldsymbol{\delta}}_{1,j}}\right) + o_N(1),
\]
where $r$ is Poisson($\gamma$)-distributed, independent of everything else.  Now clearly 
\[
\varphi(0)- \frac{\gamma}{2}\E \RM{3} = \text{\sf{Parisi}}_{\be, h, \gamma}(\zeta, K,  \mathcal V_{\boldsymbol m}) +o_N(1).
\]
Plugging this into \eqref{interpolation_calc1} settles the proof of Fact \ref{interpolation}. 
\end{proof}

Finally, some technical estimates involving Binomials.

\begin{lem}\label{binomialcomp}
Let $S$ be a Binomial$(n,p)$ random variable, then for $\alpha = n p$ we have 
$$ \E[S^3 e^{\beta S}] \leq \left( \alpha^3 e^{3\beta}+3\alpha^2 e^{2\beta}+\alpha e^{\beta}\right) \exp\left((e^\beta-1)\alpha\right) $$
\end{lem}
\begin{proof}
We set $S = \sum\limits_{i=1}^n B_i$ for $B_1,..,B_n$ independent Bernoulli($p$) random variables. Then 
$$\E S^3 e^{\beta S} = \E \left( \sum\limits_{i=1}^n B_i\right)^3 \exp\left(\beta  \sum\limits_{i=1}^n B_i\right)$$
\beq\label{binocalc1}
= \E \sum\limits_{i=1}^n\sum\limits_{j=1}^n\sum\limits_{k=1}^n \prod\limits_{l=1}^n B_i B_j B_k e^{\beta B_l}
\eeq
Here is the counting: we have at most $n^3$ terms where $i,j,k$ are all different, at most $3n^2$ terms where in $i,j,k$ two are identical and the third is different and we have $n$ term where all three are identical. Since the distribution of $\prod\limits_{l=1}^n B_i B_j B_k e^{\beta B_l}$ only depends on how many of $i,j,k$ are identical we have
\[
(\ref{binocalc1}) \leq n^3 \E \prod\limits_{l=1}^n B_1 B_2 B_3 e^{\beta B_l} + 3n^2 \E\prod\limits_{l=1}^n B_1 B_2 e^{\beta B_l}+ n\E \prod\limits_{l=1}^n B_1 e^{\beta B_l}.
\]
Estimating term by term we have for the first term
\[
\E \prod\limits_{l=1}^n B_1 B_2 B_3 e^{\beta B_l} = \left(\E B_1 e^{\beta B_1}\right)^3 \left(\E e^{\beta B_1}\right)^{n-3} = p^3e^{3\beta} \left(1+(e^\beta-1)p\right)^{n-3} 
\]
and since $(1+x)^k \leq e^{kx}$ we have 
\[
n^3 \E \prod\limits_{l=1}^n B_1 B_2 B_3 e^{\beta B_l}\leq n^3 p^3e^{3\beta} \exp\left( (e^\beta-1)p n\right). 
\]
The same calculations for the other two terms yield
\[
3n^2 \E\prod\limits_{l=1}^n B_1 B_2 e^{\beta B_l} \leq 3n^2 p^2 e^{2\beta}\exp\left( (e^\beta-1)p n\right) 
\]
\[
n\E \prod\limits_{l=1}^n B_1 e^{\beta B_l} \leq n p e^{\beta}\exp\left( (e^\beta-1)p n\right) 
\]
collecting all terms we obtain the result
\[
\E S^3 e^{\beta S} \leq \left(\left(n p e^{\beta}\right)^3+3\left(n p e^{\beta}\right)^2+\left(n p e^{\beta}\right)\right) \exp\left( (e^\beta-1)p n\right). 
\]

\end{proof}

{\bf Acknowledgments.} We warmly thank Amin Coja-Oghlan for drawing our attention to the ISP, and for useful discussions. It is also a pleasure to thank Dmitry Panchenko for shedding light on the M\'{e}zard-Parisi Ansatz, for explanations concerning his work on diluted models, and for much appreciated help with the literature.

\end{document}